\documentclass[a4paper,10pt]{article}
\usepackage[latin1]{inputenc}
\usepackage[T1]{fontenc}
\usepackage{amsmath}
\usepackage{amsthm}
\usepackage{amsfonts}
\usepackage{amstext}
\usepackage{amsthm}
\usepackage[all]{xy}
\usepackage{geometry}
\usepackage{srcltx}
\usepackage{graphics}
\usepackage{epsfig}
\usepackage{subfigure}
\usepackage{slashed}
\usepackage{color}
\usepackage{amssymb}
\usepackage{srcltx}
\usepackage{mathrsfs}
\newtheorem{theorem}{Theorem}[section]
\newtheorem{lemma}[theorem]{Lemma}
\newtheorem{defi}[theorem]{Definition}

\newtheorem{prop}[theorem]{Proposition}
\newtheorem{cor}[theorem]{Corollary}
\newtheorem*{theorem*}{Theorem}

\DeclareMathOperator{\im}{im}

\DeclareMathOperator{\sfl}{sf}
\DeclareMathOperator{\diag}{diag}

\DeclareMathOperator{\sgn}{sgn}

\DeclareMathOperator{\diverg}{div}
\DeclareMathOperator{\gra}{graph}
\DeclareMathOperator{\codim}{codim}
\DeclareMathOperator{\FredL}{\mathcal{FL}}
\DeclareMathOperator{\Sp}{Sp}

\title{A Morse-Smale index theorem for indefinite elliptic systems and bifurcation}
\author{Alessandro Portaluri and Nils Waterstraat}

\begin{document}
\date{}
\maketitle

\footnotetext[2]{A. Portaluri was partially supported by PRIN 2009 ``Critical Point Theory and Perturbative Methods for Non-Linear Differential Equations'' and by the project ERC Advanced Grant 2013 No. 339958 ``Complex Patterns for Strongly Interacting Dynamical Systems - COMPAT''.}
\footnotetext[3]{N. Waterstraat was supported by the Berlin Mathematical School and the SFB 647 ``Space--Time--Matter''.}

\begin{abstract}
We generalise the semi-Riemannian Morse index theorem to non-degenerate elliptic systems of partial differential equations on star-shaped domains. Moreover, we apply our theorem to bifurcation from a branch of trivial solutions of semilinear systems, where the bifurcation parameter is introduced by shrinking the domain to a point. This extends recent results of the authors for scalar equations. 
\end{abstract}

\section{Introduction}
The Morse index theorem is a well known result in differential geometry which relates the Morse index of a non-degenerate geodesic $\gamma$ in a Riemannian manifold $(M,g)$ to its number of conjugate points (cf. \cite[\S 15]{MilnorMorse}). It was proved by Marston Morse in the first third of the 20th century (cf. \cite{MorseTrans}, \cite{Morse}) and since then it has been generalised into various directions. After introducing coordinates, the Morse index theorem can be viewed as an assertion about Dirichlet boundary value problems for systems of ordinary differential equations of the form

\begin{equation}\label{Morseequ}
\left\{
\begin{aligned}
-u''(x)+S(x)u(x)&=\lambda\,u(x),\quad x\in[0,1],\\
u(0)=u(1)&=0
\end{aligned}
\right.
\end{equation}
where $S:[0,1]\rightarrow\mathcal{S}(k;\mathbb{R})$ is a smooth path of symmetric matrices containing curvature terms of the manifold $M$ along the geodesic, and $k$ is the dimension of $M$. If we now define the {\em Morse index} $\mu_{Morse}(\gamma)$ of the geodesic to be the number of eigenvalues $\lambda<0$ of the boundary value problem \eqref{Morseequ} counted with multiplicities and 

\begin{align}\label{Morseconj}
m(t):=\dim\{u:[0,1]\rightarrow\mathbb{R}^k:\,-u''(x)+S(x)u(x)=0,\,u(0)=u(t)=0\},
\end{align} 
then the Morse index theorem states that

\begin{align}\label{Morsetheorem}
\mu_{Morse}(\gamma)=\sum_{t\in [0,1]}{m(t)}.
\end{align}
Instants $t\in [0,1]$ such that $m(t)>0$ are called {\em conjugate} and \eqref{Morsetheorem} in particular implies that they are finite in number.\\
Smale showed in \cite{Smale} (cf. also \cite{SmaleCorr}) that an equality like \eqref{Morsetheorem} continues to be true for strongly elliptic partial differential equations as follows: let $M$ be a smooth compact manifold of dimension $n$ with non-empty boundary $\partial M$, $E$ a Riemannian vector bundle of dimension $k$ over $M$ and $\varphi_t:M\rightarrow M$ a continuous curve of smooth embeddings such that $\varphi_0=id$ and $M_s:=\varphi_s(M)\subset M_t$ for $s>t$. Let $\mathcal{L}:\Gamma_0(E)\rightarrow\Gamma(E)$ be a strongly elliptic selfadjoint differential operator of even order, where $\Gamma(E)$ denotes the space of smooth sections of $E$ and $\Gamma_0(E)$ are those elements of $\Gamma(E)$ that vanish on $\partial M$. Note that by the strong ellipticity assumption, $\mathcal{L}$ has a finite Morse index, i.e., there are only finitely many negative eigenvalues of $\mathcal{L}$ which are all of finite multiplicity. We now obtain differential operators $\mathcal{L}_t:\Gamma_0(E\mid_{M_t})\rightarrow\Gamma(E\mid_{M_t})$ by restricting $\mathcal{L}$ to $E\mid_{M_t}$ and we denote

\begin{align*}
m(t)=\dim\{u\in\Gamma_0(E_t):\mathcal{L}_tu=0\}. 
\end{align*}
Smale's theorem states that under common assumptions on the operators $\mathcal{L}$, the corresponding equality \eqref{Morsetheorem} still holds. Later Uhlenbeck (cf. \cite{Uhlenbeck}) and Swanson (cf. \cite{Swansona},\cite{Swansonb}) gave alternative proofs of Smale's result using abstract Hilbert space theory and intersection theory in symplectic Hilbert spaces, respectively. Note that the classical Morse index theorem \eqref{Morsetheorem} is an immediate consequence of Smale's result in the case $n=1$.\\
A completely different variation of Morse's classical result is inspired by physical applications and concerns the corresponding statement for geodesics in semi-Riemannian manifolds $(M,g)$, which comprise the models of space-time in general relativity theory. After introducing coordinates, the equations \eqref{Morseequ} are in this more general case

\begin{equation}\label{Morseequsemi}
\left\{
\begin{aligned}
Ju''(x)+S(x)u(x)&=\lambda\,u(x),\quad x\in[0,1],\\
u(0)=u(1)&=0,
\end{aligned}
\right.
\end{equation}
where $J$ is the diagonal matrix

\begin{align}\label{J}
J=\diag(\underbrace{-1,\ldots,-1}_{k-\nu},\underbrace{1,\ldots,1}_\nu).
\end{align}
and $\nu$ is the index of the semi-Riemannian metric $g$. Twenty years ago, Helfer explored Morse index theorems for \eqref{Morseequsemi} in \cite{Helfer} and ascertained that it is not even possible to make sense of the values involved in the classical result \eqref{Morsetheorem} in this generality: the Morse index of \eqref{Morseequsemi} is easily seen to be infinite if $\nu\neq0$, and moreover, conjugate instants may accumulate. Starting with Helfer's work, considerable amount of research has been done in order to extend the Morse-index theorem to geodesics in arbitrary semi-Riemannian manifolds (cf. \cite{PiccioneMITinSRG}). A new approach to this problem was proposed by Musso, Pejsachowicz and the first author in \cite{MPP}, where topological tools like the spectral flow and the winding number were used in order to give a meaning to the indices in \eqref{Morsetheorem} in the semi-Riemannian setting. Subsequently, the second author gave an alternative proof of this general version of the Morse theorem using the Atiyah-J\"anich bundle and $K$-theoretic methods (cf. \cite{Wa}).\\
Recently, also Smale's theorem was revisited and extended to more general boundary conditions under the additional assumptions that the manifold $M$ is a star-shaped domain $\Omega$ in some Euclidean space $\mathbb{R}^n$, the shrinking $\varphi$ is the canonical contraction, and in particular, $k=1$, i.e., only scalar partial differential equations were considered. Deng and Jones studied in \cite{DengJones} zeroth-order perturbations of the scalar Laplacian for a rather general class of boundary value problems. Subsequently, the first author extended their results for the Dirichlet and Neumann problem in collaboration with Dalbono to general scalar second order elliptic partial differential equations in \cite{AleDalbono}. The novelty in these investigations is that now, except for the case of the classical Dirichlet condition as treated by Smale in \cite{Smale}, conjugate points can accumulate as in the case of semi-Riemannian geodesics. Hence the right hand side in \eqref{Morsetheorem} does no longer exist, while the left hand side is still well defined. Deng and Jones tried to overcome this problem in \cite{DengJones} by using a Maslov index for curves of Lagrangian subspaces in a symplectic Hilbert space consisting of functions on the boundary of $\Omega$. Note that compared to \eqref{Morseequ}, the equations considered in \cite{DengJones} and \cite{AleDalbono} correspond for Dirichlet boundary conditions to the case of geodesics in one-dimensional Riemannian manifolds.\\
Finally, the authors studied bifurcation phenomena for scalar semilinear elliptic differential equations on star-shaped domains under shrinking of the domain by variational methods in \cite{AleIchDomain} and \cite{AleIchBall}, and obtained incidentally a new proof of Smale's theorem \cite{Smale} for scalar elliptic equations (cf. also \cite{WaterstraatSurvey}).\\
The aim of this work is to extend the semi-Riemannian index theorem from \cite{MPP} for the indefinite boundary value problem \eqref{Morseequsemi} to elliptic systems of partial differential equations, and to study bifurcation phenomena under shrinking of the domain. Let $\Omega$ be a smooth bounded domain in $\mathbb{R}^n$ which we assume to be star-shaped with respect to $0$. In what follows, we denote for $0<r\leq 1$

\[\Omega_r:=\{r\cdot x:\, x\in\Omega\}\]
and we consider the Dirichlet boundary value problems   

\begin{equation}\label{bvpnonlin}
\left\{
\begin{aligned}
J\Delta u(x)+V(x,u(x))&=0,\quad x\in\Omega_r\\
u(x)&=0,\quad x\in\partial\Omega_r,
\end{aligned}
\right.
\end{equation}
where $V:\overline{\Omega}\times\mathbb{R}^k\rightarrow\mathbb{R}^k$ is a smooth map such that $V(x,0)=0$ for all $x\in\Omega$ and $J$ is as in \eqref{J} for some $0\leq\nu\leq k$. We call the parameter $r$ the radius, and we note that $u\equiv 0$ is a solution of \eqref{bvpnonlin} for all $r\in(0,1]$. A radius $r^\ast\in(0,1]$ is said to be a \textit{bifurcation radius} if there exist a sequence $\{r_n\}_{n\in\mathbb{N}}$ and weak solutions $0\neq u_n\in H^1_0(\Omega_{r_n},\mathbb{R}^k)$ of \eqref{bvpnonlin} such that $r_n\rightarrow r^\ast$ and $\|u_n\|_{H^1_0(\Omega_{r_n},\mathbb{R}^k)}\rightarrow 0$. As we will see below, an important object for studying bifurcation is given by the linearisation of \eqref{bvpnonlin}, which is the Dirichlet boundary value problem

\begin{equation}\label{bvplin}
\left\{
\begin{aligned}
J\Delta u(x)+S(x)u(x)&=0,\quad x\in\Omega_r\\
u(x)&=0,\quad x\in\partial\Omega_r,
\end{aligned}
\right.
\end{equation}
where

\[S(x):=D_uV(x,0),\quad x\in\Omega,\]
is a smooth family of $k\times k$ matrices. In what follows, we assume that $S(x)$ is symmetric for all $x\in\Omega$. We call $r\in(0,1]$ a \textit{conjugate radius} if \eqref{bvplin} has a non-trivial solution, and we say that \eqref{bvpnonlin} is non-degenerate if $1$ is not a conjugate radius for \eqref{bvplin}.\\
Note that if $J=-I_k$, where $I_k$ is the identity matrix of size $k$, then \eqref{bvplin} is a special case of the equations considered by Smale in \cite{Smale}, and if moreover $k=1$, such equations were treated in \cite{DengJones} for more general boundary conditions.\\
It is worth to note that for $n=1$, i.e., a one dimensional domain $\Omega$, \eqref{bvplin} are precisely the equations \eqref{Morseequsemi} coming from geodesics in semi-Riemannian manifolds. Consequently, if $J\neq -I_k$ in this case, then the corresponding Morse index is infinite, and conjugate radii may accumulate according to Helfer's work \cite{Helfer} which we have already mentioned above. In particular, Smale's theorem \cite{Smale} cannot hold in the situation that we are studying here, and our aim is to extend it to our equations \eqref{bvplin} as the Morse index theorem \eqref{Morsetheorem} was generalised to the equations \eqref{Morseequsemi} in \cite{MPP}. Accordingly, we substitute the Morse index of the equations \eqref{bvplin} by the spectral flow $\sfl(h,[0,1])$ of a suitable path of Fredholm quadratic forms $h$ as in \cite{MPP}. Moreover, we use the Maslov index $\mu_{Mas}(\ell,\mu,[0,1])$ for paths of Lagrangian subspaces $\ell$ in the von Neumann quotient $\beta$ from \cite{BoossFurutani} to extend the right hand side in \eqref{Morsetheorem} to our equations \eqref{bvplin}, where the subspaces $\ell$ are obtained from the abstract Cauchy data spaces of \eqref{bvplin} and $\mu$ corresponds to the Dirichlet boundary condition. Our main theorem establishes a Morse index theorem for elliptic systems of second order partial differential equations which are not necessarily strongly elliptic, and reads as follows:

\begin{theorem*}
If \eqref{bvpnonlin} is non-degenerate, then

\[\sfl(h,[0,1])=\mu_{Mas}(\ell, \mu, [0,1])\in\mathbb{Z}.\]
\end{theorem*}

Moreover, we introduce in a second theorem a new proof of Smale's theorem for the equations \eqref{bvplin} in the strongly elliptic case, i.e., if $J=-I_k$.

\begin{theorem*}
If \eqref{bvpnonlin} is non-degenerate and $J=-I_k$, then there are only finitely many conjugate instants in $(0,1)$ and

\[\sfl(h,[0,1])=-\sum_{r\in(0,1)}{m(r)},\]
where $m(r)$ denotes the dimension of the space of classical solutions of \eqref{bvplin}.
\end{theorem*}

Finally, the case $n=1$ is considered in which \eqref{bvplin} are ordinary differential equations. We derive as immediate corollary the classical Morse index theorem \eqref{Morsetheorem}, and we also prove that the semi-Riemannian Morse index theorem \cite{MPP} follows from our main theorem on systems of partial differential equations.\\
Subsequently, we use our index theorems to discuss the relation between conjugate radii and bifurcation radii. Our results extend the main theorems of the recent articles \cite{AleIchDomain} and \cite{AleIchBall} of the authors which show that for scalar equations, i.e., $k=1$ and $J=-I_1$, conjugate radii and bifurcation radii coincide. We will see below that the same assertion holds for $k>1$ as long as $J=-I_k$, however, the remarkable difference is that for $J\neq-I_k$ the bifurcation radii are in general just a proper subset of the conjugate radii. We illustrate this below by an example.\\
The paper is structured as follows: in the second section we consider the weak formulation of the equations \eqref{bvpnonlin} and we introduce the generalised Morse index, which is defined by means of the spectral flow for paths of Fredholm quadratic forms that we introduce before in a separate section. In the third section we define the Maslov index for \eqref{bvplin}, where we follow the ideas of Booss and Furutani from \cite{BoossFurutani}. In the fourth section we state and prove our main theorems on the equality of these indices and their corollaries. In the fifth section we discuss bifurcation for \eqref{bvpnonlin} under shrinking of the domain in connection with the non-vanishing of the indices for the linearised equations \eqref{bvplin}. Finally, there are two appendices. In the first one we recall the definition of the spectral flow for paths of selfadjoint Fredholm operators and crossing forms from \cite{RobbinSalamon} which are important in our proofs.  In the second one we recall some facts about the Fredholm Lagrangian Grassmannian of symplectic Hilbert spaces and the Maslov index, where we follow Furutani's survey \cite{Furutani}.

\subsubsection*{Acknowledgements}
The authors wish to express their thanks to Graham Cox, Christopher K.R.T. Jones and Yuri Latushkin for several helpful conversations. Moreover, we are grateful to Jacobo Pejsachowicz and the anonymous referee for many valuable comments which have improved the presentation of this work.


\section{The generalised Morse index}
In this section we recall in a first subsection the spectral flow for Fredholm quadratic forms that was introduced in \cite[\S2]{MPP}. Subsequently, we consider the weak formulations of the equations \eqref{bvpnonlin} and define the generalised Morse index of the linearised equation \eqref{bvplin}.

\subsection{The spectral flow for Fredholm quadratic forms}\label{sect-sfl}
In what follows, we assume that the reader is familiar with the definition of the spectral flow for paths of selfadjoint Fredholm operators, which we recap in Appendix \ref{sect-appsfl}.\\
A bounded quadratic form $q: H\rightarrow\mathbb{R}$ on a real Hilbert space $H$ is a map for which there exists a bounded symmetric bilinear form $b_{q}\colon H\times H\rightarrow\mathbb{R}$ such that $q(u)=b_q(u,u)$, $u\in H$. By the Riesz representation theorem, there is a bounded selfadjoint operator $L_q:H\rightarrow H$ such that 

\begin{align}\label{Rieszrep}
b_q(u,v)=\langle L_qu,v\rangle_H,\quad u,v\in H.
\end{align}
We call $L_q$ the Riesz representation of $q$, and $q$ is a {\em Fredholm quadratic form\/} if $L_q$ is Fredholm, i.e., $\ker L_q$ is of finite dimension and $\im L_q$ is closed.\\
The space $Q(H)$ of bounded quadratic forms is a Banach space with respect to the norm 

\[ \left\|  q\right\| = \sup\left\{ \left| q(u)\right| : \left\| u\right\| =1\right\}=\|L_q\|.\]
The subset $Q_F(H)$ of all Fredholm quadratic forms is an open subset of $Q(H)$ which is stable under perturbations by weakly continuous quadratic forms (cf. \cite[\S 21.10]{Zeidler}). A quadratic form $q\in Q(H)$ is called {\em non-degenerate\/} if the corresponding Riesz representation $L_q$ is invertible. Since the set $GL(H)$ of invertible operators is open in $\mathcal{L}(H)$, the non-degenerate quadratic forms are open in $Q(H)$ as well. The \textit{Morse index} of $q\in Q_F(H)$ is defined by 

\begin{align}\label{Morse}
\begin{split}
\mu_{Morse}(q)&:=\sup\{\dim U:\,U\,\text{is a linear subspace of}\ H,\,q(u)<0\,\,\text{for all}\,\, u\in U\setminus\{0\}\}\\
&=\mu_{Morse}(L_q),
\end{split}
\end{align}
where the latter equality easily follows from functional calculus (cf. \cite[Prop. 9.4.2]{PalaisTerng}). As for bounded selfadjoint Fredholm operators, the space $Q_F(H)$ consists of three components 

\[Q_F(H)=Q^+_F(H)\cup Q^i_F(H)\cup Q^-_F(H),\]
where $Q^\pm_F(H)=\{q\in Q_F(H):\, \mu_{Morse}(\pm q)<\infty\}$ are contractible and $Q^i_F(H)$ is a classifying space for the $KO$-theory functor $KO^{-7}$.

\begin{defi}
Let $q:[a,b]\rightarrow Q_F(H)$ be a path having non-degenerate endpoints. The spectral flow of $q$ is defined by

\[\sfl(q,[a,b])=\sfl(L_q,[a,b]).\]
\end{defi}

The following properties of the spectral flow are immediate consequences of the corresponding assertions in Appendix A.

\begin{enumerate}
	\item[i)] If $h:[0,1]\times[a,b]\rightarrow Q_F(H)$ is such that $h(\lambda,a)$ and $h(\lambda,b)$ are non-degenerate for all $\lambda\in[0,1]$, then
	
	\[\sfl(h(0,\cdot),[a,b])=\sfl(h(1,\cdot),[a,b]).\]
	
	\item[ii)] If $q_t$ is non-degenerate for all $t\in[a,b]$, then 
	
	\[\sfl(q,[a,b])=0.\]
	
	\item[iii)] If $q_c$ is non-degenerate for some $c\in(a,b)$, then
	
	\[\sfl(q,[a,b])=\sfl(q,[a,c])+\sfl(q,[c,b]).\]
	
	\item[iv)] If $q_t\in Q^+_F(H)$ for all $t\in[a,b]$, then
	
	\[\sfl(q,[a,b])=\mu_{Morse}(q_a)-\mu_{Morse}(q_b).\]   

\end{enumerate}

As for paths of selfadjoint Fredholm operators (cf. Thm. \ref{thm:crossing}), the spectral flow can be computed explicitly for paths of Fredholm quadratic forms which are sufficiently regular. If $q\colon [a,b]\rightarrow Q_F(H)$ is differentiable at $t$, then the derivative $\dot{q}(t)$ with respect to $t$ is again a quadratic form. We call $t\in[a,b]$ a crossing if $q(t)$ is degenerate and we say that $t$ is regular if the {\em crossing form\/} $\Gamma(q,t)$, defined by 

\[\Gamma(q,t):=\dot{q}(t)|_{\ker L_{q(t)}},
\]
is non-degenerate.

\begin{prop}\label{crossformcomsfl}
We assume that $q:[a,b]\rightarrow Q_F(H)$ is continuously differentiable and has non-degenerate endpoints. If all crossings $t$ of $q$ are regular, then they are finite
in number and

\begin{equation*}
\sfl(q,[a,b])= \sum_{t\in(a,b)}  \sgn \, \Gamma(q,t).
\end{equation*}
\end{prop}


\subsection{The generalised Morse index and conjugate points}\label{section:genMorse}
Let us assume as in the introduction that $\Omega$ is a smooth bounded domain in $\mathbb{R}^n$ which is star-shaped with respect to $0$, and let $J$ be a diagonal matrix as in \eqref{J} for some $0\leq\nu\leq k$. We define a function $\nu:\{0,\ldots,k\}\rightarrow\mathbb{Z}_2$ by

\begin{align*}
\nu(j)=\begin{cases}
1\quad\text{if}\,\, 1\leq j\leq k-\nu  \\
-1\quad\text{if}\,\, k-\nu<j\leq k
\end{cases}
\end{align*}
Let $V=(V^1,\ldots,V^k):\overline{\Omega}\times\mathbb{R}^k\rightarrow\mathbb{R}^k$ be a smooth gradient vector field, i.e., there exists some $G:\overline{\Omega}\times\mathbb{R}^k\rightarrow\mathbb{R}$ such that $\nabla_2G(x,\xi)=V(x,\xi)$ for all $x\in\Omega$ and $\xi\in\mathbb{R}^k$, where $\nabla_2$ denotes the gradient with respect to the variable in $\mathbb{R}^k$. In what follows, we suppose that there are constants $\alpha, C$ such that for $j=1,\ldots,k$

\begin{align}\label{growth}
|\nabla_2V^j(x,\xi)|\leq C(1+|\xi|^{\alpha-1}),\quad (x,\xi)\in\Omega\times\mathbb{R}^k, 
\end{align}
where $1\leq\alpha\leq\frac{n+2}{n-2}$ if $n\geq 3$ and $1\leq\alpha<\infty$ if $n=2$. Finally, in the case $n=1$, that is, \eqref{bvpnonlin} is an ordinary differential equation, we do not impose a growth condition on the nonlinearity $V$.\\
It is well known that under the assumption \eqref{growth}, the functional

\begin{align*}
\psi:H^1_0(\Omega,\mathbb{R}^k)\rightarrow\mathbb{R},\quad \psi(u)=-\frac{1}{2}\sum^k_{j=1}{(-1)^{\nu(j)}\int_\Omega{\langle\,\nabla u^j,\nabla u^j\rangle\,dx}}+\int_\Omega{G(x,u(x))\,dx}
\end{align*}
is $C^2$ and its derivative at $u\in H^1_0(\Omega,\mathbb{R}^k)$ is

\begin{align*}
(D_u\psi)(v)=-\sum^k_{j=1}(-1)^{\nu(j)}{\int_\Omega{\langle \,\nabla u^j,\nabla v^j\rangle\,dx}}+\int_\Omega{\langle V(x,u(x)),v(x)\rangle\,dx},\quad v\in H^1_0(\Omega,\mathbb{R}^k).
\end{align*}
Consequently, the critical points of $\psi$ are precisely the weak solutions of the non-linear equation \eqref{bvpnonlin} on the domain $\Omega=\Omega_1$.\\
From now on, we assume that 

\begin{align}\label{Vvanishes}
V(x,0)=0,\quad x\in\Omega,
\end{align}
which implies that $0\in H^1_0(\Omega,\mathbb{R}^k)$ is a critical point of $\psi$. The Hessian of $\psi$ at $0$ is the bilinear form

\begin{align*}
D^2_0\psi(u,v)=-\sum^k_{j=1}{(-1)^{\nu(j)}\int_\Omega{\langle \,\nabla u^j,\nabla v^j\rangle\,dx}}+\int_\Omega{\langle S(x)u(x),v(x)\rangle\,dx},\quad u,v\in H^1_0(\Omega,\mathbb{R}^k),
\end{align*}
where  

\[S(x)=(D_uV)(x,0),\quad x\in\Omega.\]
Note that $S(x)$ is symmetric since it is the Hessian matrix of $G(x,\cdot):\mathbb{R}^k\rightarrow\mathbb{R}^k$ at the critical point $0\in\mathbb{R}^k$.\\
If we now set as in the introduction for $r\in(0,1]$

\[\Omega_r:=\{r\cdot x:\,x\in\Omega\},\]
then it is readily seen from \eqref{Vvanishes} that $0\in H^1_0(\Omega_r,\mathbb{R}^k)$ is a critical point of all functionals $\psi_r:H^1_0(\Omega_r,\mathbb{R}^k)\rightarrow\mathbb{R}$ defined by

\begin{align}\label{psir}
\psi_r(u)=-\frac{1}{2}\sum^k_{j=1}{(-1)^{\nu(j)}\int_{\Omega_r}{\langle \,\nabla u^j,\nabla u^j\rangle\,dx}}+\int_{\Omega_r}{G(x,u(x))\,dx},
\end{align}
and the corresponding Hessians are given by

\[D^2_0\psi_r(u,v)=-\sum^k_{j=1}{(-1)^{\nu(j)}\int_{\Omega_r}{\langle \,\nabla u^j,\nabla v^j\rangle\,dx}}+\int_{\Omega_r}{\langle S(x)u(x),v(x)\rangle\,dx},\quad u,v\in H^1_0(\Omega_r,\mathbb{R}^k).\]
After a change of coordinates $x\mapsto r\cdot x$, these transform to

\begin{align*}
-&\frac{1}{2}\sum^k_{j=1}{(-1)^{\nu(j)}r^n\int_{\Omega}{\langle \,\nabla u^j(r\cdot x),\nabla u^j(r\cdot x)\rangle\,dx}}+r^n\int_{\Omega}{G(r\cdot x,u(r\cdot x))\,dx}\\
=&-\frac{1}{2}\sum^k_{j=1}{(-1)^{\nu(j)}r^{n-2}\int_{\Omega}{\langle \,\nabla u^j_r(x),\nabla u^j_r(x)\rangle\,dx}}+r^n\int_{\Omega}{G(r\cdot x,u_r(x))\,dx}
\end{align*}
and

\begin{align*}
-&\sum^k_{j=1}{(-1)^{\nu(j)}r^{n} \int_\Omega{\langle \,\nabla u^j(r\cdot x),\nabla v^j(r\cdot x)\rangle\,dx}}+r^n\int_\Omega{\langle S(r\cdot x)u(r\cdot x),v(r\cdot x)\rangle\,dx}\\
=&-\sum^k_{j=1}{(-1)^{\nu(j)}r^{n-2} \int_\Omega{\langle \,\nabla u^j_r(x),\nabla v^j_r(x)\rangle\,dx}}+r^n\int_\Omega{\langle S(r\cdot x)u_r(x),v_r(x)\rangle\,dx},
\end{align*}
where $u_r(x)=u(r\cdot x)$ and $v_r(x)=v(r\cdot x)$ for $x\in\Omega$.\\
We now define a family of functionals $\tilde{\psi}:[0,1]\times H^1_0(\Omega,\mathbb{R}^k)\rightarrow\mathbb{R}$ by

\begin{align}\label{tildepsi}
\tilde{\psi}_r(u)=-\frac{1}{2}\sum^k_{j=1}{(-1)^{\nu(j)}\int_{\Omega}{\langle \,\nabla u^j(x),\nabla u^j(x)\rangle\,dx}}+r^2\int_{\Omega}{G(r\cdot x,u(x))\,dx},
\end{align}
and quadratic forms $h:[0,1]\times H^1_0(\Omega,\mathbb{R}^k)\rightarrow\mathbb{R}$ by

\begin{align*}
h_r(u)=-\sum^k_{j=1}{(-1)^{\nu(j)}\int_\Omega{\langle \,\nabla u^j,\nabla u^j\rangle\,dx}}+\int_\Omega{\langle S_r( x)u(x),u(x)\rangle\,dx},\quad u\in H^1_0(\Omega,\mathbb{R}^k),\, r\in[0,1],
\end{align*}
where $S_r(x):=r^2S(r\cdot x)$. Note that $h_r(u)=D^2_0\tilde{\psi}_r(u,u)$, $u\in H^1_0(\Omega,\mathbb{R}^k)$, $r\in[0,1]$.

\begin{lemma}\label{lem:Fredholm}
$h_r=h_0+r^2c_r$, $r\in[0,1]$, where $h_0$ is a non-degenerate Fredholm quadratic form and $c_r$ is weakly continuous. In particular, $h$ is a path of Fredholm quadratic forms on $H^1_0(\Omega,\mathbb{R}^k)$. 
\end{lemma}

\begin{proof}
We note at first that 

\[h_0(u)=-\sum^k_{j=1}{(-1)^{\nu(j)}\int_\Omega{\langle \,\nabla u^j,\nabla u^j\rangle\,dx}},\quad u\in H^1_0(\Omega,\mathbb{R}^k),\]
is a non-degenerate Fredholm quadratic form on $H^1_0(\Omega,\mathbb{R}^k)$, which is a simple consequence of the well known Poincar\'{e} inequality.\\
For the weak continuity of $c_r$, let $\{u_n\}_{n\in\mathbb{N}}$ be a sequence in $H^1_0(\Omega,\mathbb{R}^k)$ which weakly converges to some $u\in H^1_0(\Omega,\mathbb{R}^k)$. By the compactness of the embedding $H^1_0(\Omega,\mathbb{R}^k)\subset L^2(\Omega,\mathbb{R}^k)$, $\{u_n\}_{n\in\mathbb{N}}$ converges strongly in $L^2(\Omega,\mathbb{R}^k)$ to $u$. Since $c_r$ extends to a bounded quadratic form on $L^2(\Omega,\mathbb{R}^k)$, we conclude that $c_r(u_n)\rightarrow c_r(u)$.
\end{proof}

We leave the proof of the following elementary lemma to the reader.

\begin{lemma}\label{lemmaconj} 
The following assertions are equivalent for $r\in(0,1]$:  

\begin{enumerate}
\item $h_r$ is degenerate;
\item there exists $u\in H^1_0(\Omega_r,\mathbb{R}^k)$ such that $D^2_0\psi_r(u,v)=0$ for all $v\in H^1_0(\Omega_r,\mathbb{R}^k)$;
\item $r$ is a conjugate radius, i.e the boundary value problem

\begin{equation*}
\left\{
\begin{aligned}
J\Delta u(x)+S(x)u(x)&=0,\quad x\in\Omega_r\\
u(x)&=0,\quad x\in\partial\Omega_r,
\end{aligned}
\right.
\end{equation*}
that we already introduced in \eqref{bvplin} as linearisation of \eqref{bvpnonlin}, has a non-trivial solution.
\end{enumerate}
\end{lemma} 
Let us now assume that $1$ is not a conjugate radius, which implies that the quadratic form $h_1$ is non-degenerate. Since we see from Lemma \ref{lem:Fredholm} that $h_0$ is non-degenerate as well, the spectral flow $\sfl(h,[0,1])$ is defined, to which we refer as the generalised Morse index because of the following lemma.

\begin{lemma}\label{Morse=sfl}
If $J=-I_k$, then $\sfl(h,[0,1])=-\mu_{Morse}(h_1)$, which moreover is the number of eigenvalues $\lambda<0$ of 

\begin{equation}\label{bvplindef}
\left\{
\begin{aligned}
-\Delta u(x)+S(x)u(x)&=\lambda\,u(x),\quad x\in\Omega\\
u(x)&=0,\quad x\in\partial\Omega,
\end{aligned}
\right.
\end{equation}
counted with multiplicities.
\end{lemma}

\begin{proof}
If $J=-I_k$, then $h_0(u)>0$ for all $0\neq u\in H$, and consequently $h_r\in Q^+_F(H)$ since by Lemma \ref{lem:Fredholm} there is a path in $Q_F(H)$ joining $h_0$ and $h_r$. We obtain from property iv) in Section \ref{sect-sfl}

\[\sfl(h,[0,1])=\mu_{Morse}(h_0)-\mu_{Morse}(h_1)=-\mu_{Morse}(h_1),\]
where we use again that $h_0$ is positive definite. For the remaining claim, we argue similar as in \cite[Lemma 1.1]{Duistermaat} and let $\alpha>0$ be such that $\alpha I_k+S(x)$ is positive for all $x\in\Omega$. Then

\[\langle u,v\rangle_\alpha:=\sum^k_{j=1}{\int_\Omega{\langle \,\nabla u^j,\nabla v^j\rangle\,dx}}+\int_\Omega{\langle (\alpha I_k+S(x))u(x),v(x)\rangle\,dx},\quad u,v\in H^1_0(\Omega,\mathbb{R}^k),\]
is a scalar product on $H^1_0(\Omega,\mathbb{R}^k)$ which is equivalent to $\langle\cdot,\cdot\rangle_{H^1_0(\Omega,\mathbb{R}^k)}$. It is readily seen that the equality in \eqref{Morse} remains true if we replace the scalar product on $H$ by an equivalent one, and so $\mu_{Morse}(h_1)$ is equal to the total number of negative eigenvalues of the operator $L$ determined by

\begin{align*}
\langle Lu,u\rangle_\alpha=h_1(u)=\sum^k_{j=1}{\int_\Omega{\langle \,\nabla u^j,\nabla u^j\rangle\,dx}}+\int_\Omega{\langle S(x)u(x),u(x)\rangle\,dx},\quad u\in H^1_0(\Omega,\mathbb{R}^k). 
\end{align*}    
However, $\lambda<0$ is an eigenvalue of $L$ if and only if $\frac{\lambda\alpha}{1-\lambda}<0$ is an eigenvalue of \eqref{bvplindef}, which completes the proof. 
\end{proof}

Finally, let us point out that $\mu_{Morse}(h_1)=\infty$ if $J\neq -I_k$. Indeed, if $\nu\neq 0$, then clearly there exists an infinite dimensional subspace of $H^1_0(\Omega,\mathbb{R}^k)$ on which $h_0$ is negative definite. Consequently, $h_0\in Q^i_F(H)\cup Q^-_F(H)$, and moreover we see from Lemma \ref{lem:Fredholm} that $h_1$ lies in the same path component of $Q_F(H)$ than $h_0$.


\section{The Maslov index}\label{sect:Maslov}
As we have pointed out already in the introduction, it follows from Helfer's work \cite{Helfer} for geodesics in semi-Riemannian manifolds that conjugate radii of our equations \eqref{bvplin} may accumulate if $J\neq-I_k$, and hence they cannot just be counted as in \cite{Smale}. In this section we use a construction from \cite{BoossFurutani} to assign a Maslov-type index to the family of equations \eqref{bvplin}, which can be interpreted as a generalised counting of conjugate radii.\\  
For this purpose, we first need to introduce the von Neumann quotient of our equations (cf. \cite[\S XII.2]{DunfordSchwarz}) as a symplectic Hilbert space. Let us recall that the vectorial Laplacian 

\[\Delta:C^\infty(\Omega,\mathbb{R}^k)\rightarrow C^\infty(\Omega,\mathbb{R}^k),\quad \Delta u=(\Delta u^1,\ldots,\Delta u^k)\]
is closed and symmetric on the domain

\[D_{min}:=H^2_0(\Omega,\mathbb{R}^k)=\{u=(u^1,\ldots,u^k)\in H^2(\Omega,\mathbb{R}^k):u^j\mid_{\partial\Omega}=\partial_nu^j=0,\, j=1,\ldots,k\},\]
where

\begin{align}\label{normalder}
\partial_nu^j(x)=\sum^n_{i=1}{\frac{\partial u^j}{\partial x_i}(x)\nu^i(x)},\quad x\in\partial\Omega,
\end{align}
and $\nu=(\nu^1,\ldots,\nu^n)$ is the outward pointing normal vector to the boundary of $\Omega$, and $\partial_n$ the derivative with respect to $\nu$. In what follows, we denote by $\Delta_{J}$ the restriction of $J\Delta$ to $D_{min}$, and we let $\Delta^\ast_{J}$ be the adjoint of $\Delta_{J}$, i.e., the unique linear operator on the domain

\[D(\Delta^\ast_J)=\{v\in L^2(\Omega,\mathbb{R}^k):\, u\mapsto\langle\Delta_Ju,v\rangle_{L^2(\Omega,\mathbb{R}^k)}\, \text{is bounded on}\, D_{min}\}\]
determined by 

\[\langle\Delta_Ju,v\rangle_{L^2(\Omega,\mathbb{R}^k)}=\langle u,\Delta^\ast_Jv\rangle_{L^2(\Omega,\mathbb{R}^k)},\quad u\in D_{min},\quad v\in D(\Delta^\ast_J).\] 
It is well known (cf. \cite{Grubb}) that $\Delta^\ast_{J}$ is given by the operator $J\Delta$ on the domain

\[D(\Delta^\ast_J)=D_{max}:=\{u\in L^2(\Omega,\mathbb{R}^k):\, \Delta u\in L^2(\Omega,\mathbb{R}^k)\}.\]
If we consider on the latter space the graph scalar product

\[\langle u,v\rangle_{\Delta^\ast_J}=\langle u,v\rangle_{L^2(\Omega,\mathbb{R}^k)}+\langle \Delta^\ast_J u,\Delta^\ast_J v\rangle_{L^2(\Omega,\mathbb{R}^k)},\quad u,v\in D_{max},\]
then it is a Hilbert space and $D_{min}$ is a closed subspace of it. Consequently, the quotient space $\beta:=D_{max}/D_{min}$ is a Hilbert space, which is called the \textit{von Neumann space} of $\Delta_J$. In what follows, we denote by $\tau$ the quotient map from $D_{max}$ to $\beta$. We define a bilinear form on $\beta$ by

\[\omega:\beta\times\beta\rightarrow\mathbb{R},\quad \omega(\tau(u),\tau(v))=\langle \Delta^\ast_J u,v\rangle_{L^2(\Omega,\mathbb{R}^k)}-\langle u, \Delta^\ast_J v\rangle_{L^2(\Omega,\mathbb{R}^k)},\]
and we point out for later reference that for $u,v\in H^2(\Omega,\mathbb{R}^k)\subset D_{max}$

\begin{align}\label{omegaregular}
\omega(\tau[u],\tau[v])=\sum^k_{j=1}{(-1)^{\nu(j)}\int_{\partial\Omega}{ (\partial_nu^j)(x)\,v^j(x)\,dx}}-\sum^k_{j=1}{(-1)^{\nu(j)}\int_{\partial\Omega}{ u^j(x)\,(\partial_nv^j)(x)\,dx}},
\end{align}
which follows from integration by parts. In particular, $\omega$ vanishes on $H^2_0(\Omega,\mathbb{R}^k)$ and so it is well defined on $\beta$. From now on, we assume that the reader is familiar with the fundamental notions of symplectic Hilbert spaces and the Maslov index as presented in Appendix \ref{App:Maslov}. Let us recall the following lemma from \cite[\S 3.1]{BoossFurutani} in our setting for the convenience of the reader.

\begin{lemma}
$\omega$ is a symplectic form on $\beta$.
\end{lemma}

\begin{proof}
Let us first point out that $\omega$ is skew-symmetric by definition, and its boundedness is readily seen from an elementary estimate. It remains to show that $\omega$ is non-degenerate.\\
As usual, we can identify $\beta$ with the orthogonal complement $D^\perp_{min}$ of $D_{min}$ in $D_{max}$ with respect to $\langle\cdot,\cdot\rangle_{\Delta^\ast_J}$, and now we first claim that

\begin{align}\label{Dperp}
D^\perp_{min}=\{u\in D_{max}:\, \Delta^\ast_Ju\in D_{max},\, (\Delta^\ast_J)^2u=-u\}.
\end{align}
Indeed, if we assume that $u\in D^\perp_{min}$, then 

\begin{align}\label{compomega}
\begin{split}
0&=\langle u,v\rangle_{L^2(\Omega,\mathbb{R}^k)}+\langle \Delta^\ast_Ju,\Delta^\ast_Jv\rangle_{L^2(\Omega,\mathbb{R}^k)}\\
&=\langle u,v\rangle_{L^2(\Omega,\mathbb{R}^k)}+\langle \Delta^\ast_Ju,\Delta_Jv\rangle_{L^2(\Omega,\mathbb{R}^k)},\quad v\in D_{min}.
\end{split}
\end{align} 
Consequently, $v\mapsto\langle\Delta^\ast_Ju,\Delta_Jv\rangle_{L^2(\Omega,\mathbb{R}^k)}$ is bounded on $D_{min}$ and so $\Delta^\ast_Ju\in D_{max}$. Moreover, we see that

\[\langle \Delta^\ast_Ju,\Delta_Jv\rangle_{L^2(\Omega,\mathbb{R}^k)}=\langle\Delta^\ast_J(\Delta^\ast_Ju),v\rangle_{L^2(\Omega,\mathbb{R}^k)}\]
and so \eqref{compomega} gives $\langle v,\Delta^\ast_J(\Delta^\ast_Ju)-u\rangle_{L^2(\Omega,\mathbb{R}^k)}=0$ for all $v\in D_{min}$. Since $D_{min}$ is dense in ${L^2(\Omega,\mathbb{R}^k)}$, this shows that $\Delta^\ast_J(\Delta^\ast_Ju)=-u$.\\
Conversely, let us assume that $u\in D_{max}$ belongs to the set on the right hand side of equation \eqref{Dperp} and that $v\in D_{min}$. Then

\[\langle u,v\rangle_{L^2(\Omega,\mathbb{R}^k)}+\langle\Delta^\ast_Ju,\Delta^\ast_Jv\rangle_{L^2(\Omega,\mathbb{R}^k)}=\langle u,v\rangle_{L^2(\Omega,\mathbb{R}^k)}+\langle\Delta^\ast_Ju,\Delta_Jv\rangle_{L^2(\Omega,\mathbb{R}^k)}=\langle u-u,v\rangle_{L^2(\Omega,\mathbb{R}^k)}=0\]
and so $u\in D^\perp_{min}$. This finishes the proof of \eqref{Dperp}.\\
We next claim that the restriction of $\Delta^\ast_J$ to $D^\perp_{min}$ is an automorphism. If $v:=\Delta^\ast_Ju$ for some $u\in D^\perp_{min}$, then clearly $v\in D_{max}$, $\Delta^\ast_Jv=-u\in D^\perp_{min}$ and $\Delta^\ast_J(\Delta^\ast_Jv)=-\Delta^\ast_Ju=-v$. Consequently, $v\in D^\perp_{min}$, and we conclude that $\Delta^\ast_J:D^\perp_{min}\rightarrow D^\perp_{min}$ is an isomorphism having as inverse $(\Delta^\ast_J)^{-1}=-\Delta^\ast_J$.\\
Finally, we note that for $u,v\in D^\perp_{min}$

\begin{align*}
\omega(\tau(u),\tau(v))&=\langle\Delta^\ast_Ju,v\rangle_{L^2(\Omega,\mathbb{R}^k)}-\langle u,\Delta^\ast_Jv\rangle_{L^2(\Omega,\mathbb{R}^k)}\\
&=\langle\Delta^\ast_Ju,v\rangle_{L^2(\Omega,\mathbb{R}^k)}+\langle\Delta^\ast_J\Delta^\ast_Ju,\Delta^\ast_Jv\rangle_{L^2(\Omega,\mathbb{R}^k)}\\
&=\langle \Delta^\ast_Ju,v\rangle_{\Delta^\ast_J},
\end{align*}
which shows that $\omega$ is non-degenerate on $\beta$, since $\Delta^\ast_J:D^\perp_{min}\rightarrow D^\perp_{min}$ is an isomorphism.
\end{proof}

In what follows, we use without further reference that the restriction of $\Delta^\ast_J$ to $H^2(\Omega,\mathbb{R}^k)\cap H^1_0(\Omega,\mathbb{R}^k)$ is a selfadjoint Fredholm operator with a purely discrete spectrum (cf. eg. \cite[Thm. 2.5.7]{Frey}).\\
We now set 

\[\mu:=\tau(H^2(\Omega,\mathbb{R}^k)\cap H^1_0(\Omega,\mathbb{R}^k))\subset\beta\]
and we claim that $\mu$ is a Lagrangian subspace. Indeed, we first see by \eqref{omegaregular} that $\mu$ is isotropic, i.e., $\mu\subset\mu^\circ$. On the other hand, if $\tau(u)\in\mu^\circ$ for some $u\in D_{max}$, then $\langle\Delta^\ast_Ju,v\rangle_{L^2(\Omega,\mathbb{R}^k)}=\langle u,\Delta^\ast_Jv\rangle_{L^2(\Omega,\mathbb{R}^k)}$ for all $v\in H^2(\Omega,\mathbb{R}^k)\cap H^1_0(\Omega,\mathbb{R}^k)$. Consequently, $v\mapsto\langle \Delta^\ast_J v,u\rangle_{L^2(\Omega,\mathbb{R}^k)}$ is bounded on $H^2(\Omega,\mathbb{R}^k)\cap H^1_0(\Omega,\mathbb{R}^k)$, which shows that $u\in H^2(\Omega,\mathbb{R}^k)\cap H^1_0(\Omega,\mathbb{R}^k)$ by the selfadjointness of the restriction of $\Delta^\ast_J$ to $H^2(\Omega,\mathbb{R}^k)\cap H^1_0(\Omega,\mathbb{R}^k)$. Hence $\tau(u)\in\mu$ and so $\mu^\circ\subset\mu$.\\
In what follows, we denote by a slight misuse of notation by $S_r$, $r\in[0,1]$, the bounded selfadjoint operator on $L^2(\Omega,\mathbb{R}^k)$ defined by

\[S_r:L^2(\Omega,\mathbb{R}^k)\rightarrow L^2(\Omega,\mathbb{R}^k),\quad (S_ru)(x)=r^2\, S(r\cdot x)u(x).\]
We consider the subspaces 

\[\ell(r)=\tau(\ker(\Delta^\ast_J+S_r))\subset\beta,\quad r\in[0,1],\]
and note that for $\tau(u),\tau(v)\in\ell(r)$, we have

\begin{align*}
\omega(\tau(u),\tau(v))&=\langle \Delta^\ast_J u,v\rangle_{L^2(\Omega,\mathbb{R}^k)}-\langle u, \Delta^\ast_J v\rangle_{L^2(\Omega,\mathbb{R}^k)}\\
&=\langle \Delta^\ast_J u,v\rangle_{L^2(\Omega,\mathbb{R}^k)}+\langle S_r u,v\rangle_{L^2(\Omega,\mathbb{R}^k)}\\
&-\langle u, \Delta^\ast_J v\rangle_{L^2(\Omega,\mathbb{R}^k)}-\langle u, S_r v\rangle_{L^2(\Omega,\mathbb{R}^k)}=0,
\end{align*}
which shows that $\ell(r)$ is an isotropic subspace of $\beta$. The task is now to define the Maslov index of the curve $\ell$ with respect to $\mu$. Before, let us note for later reference the \textit{unique continuation property} of the equations \eqref{bvplin}, which states that

\begin{align}\label{ucp}
D_{min}\cap\ker(\Delta^\ast_J+S_r)=\{0\},
\end{align}
or in other words, any solution of the Dirichlet boundary problem \eqref{bvplin} such that all normal derivatives \eqref{normalder} at the boundary are trivial, has to vanish on all of $\Omega$ (cf. \cite{Calderon}, \cite{Hormander}).

\begin{prop}\label{contell}
Each subspace $\ell(r)$, $r\in[0,1]$, belongs to $\mathcal{FL}_{\mu}(\beta)$, and the path $\ell:[0,1]\rightarrow\mathcal{FL}_{\mu}(\beta)$ is smooth.
\end{prop}

\begin{proof}
A complete proof of this proposition can be found in Proposition 3.5 and Theorem 3.8. of \cite{BoossFurutani} and here we do not want to repeat the argument that $\ell(r)$ is a Lagrangian subspace of $\beta$ and that $(\ell(r),\mu)$ is a Fredholm pair, which is elementary but rather technical. Instead we want to discuss the smoothness of the curve $\ell$ in $\mathcal{FL}_{\mu}(\beta)$ since this is a crucial point that was neglected in the approach to the Maslov index of \cite{DengJones} and \cite{AleDalbono}.\\
Let us fix an $r_0\in (0,1)$ and consider the map

\[F_r:D_{max}\rightarrow L^2(\Omega,\mathbb{R}^k)\oplus\ker(\Delta^\ast_{J}+S_{r_0}),\quad u\mapsto((\Delta^\ast_{J}+S_{r})u,P_{r_0}u),\] 
where $P_{r_0}$ denotes the orthogonal projection onto $\ker(\Delta^\ast_{J}+S_{r_0})$ in $L^2(\Omega,\mathbb{R}^k)$. Note that $\ker(\Delta^\ast_{J}+S_{r_0})$ is closed in $L^2(\Omega,\mathbb{R}^k)$ as it is the kernel of a closed operator. We claim that $F_{r_0}$ is bijective. Clearly, $F_{r_0}(u)=0$ implies $u\in \ker(\Delta^\ast_{J}+S_{r_0})$ and so $0=P_{r_0}u=u$. For the surjectivity of $F_{r_0}$, we first note that

\[(\im(\Delta^\ast_{J}+S_{r_0}))^\perp=(\im(\Delta_J+S_{r_0})^\ast)^\perp=\ker(\Delta_J+S_{r_0})=0,\] where we use the unique continuation property \eqref{ucp}. From this and the fact that the restriction of $\Delta^\ast_{J}+S_{r_0}$ to $H^2(\Omega,\mathbb{R}^k)\cap H^1_0(\Omega,\mathbb{R}^k)$ is Fredholm, it follows that $\im(\Delta^\ast_{J}+S_{r_0})=L^2(\Omega,\mathbb{R}^k)$. If now $y\in L^2(\Omega,\mathbb{R}^k)$ and $x\in\ker(\Delta^\ast_{J}+S_{r_0})$, then there exists $z\in D_{max}$ such that $(\Delta^\ast_{J}+S_{r_0})z=y$. We set $w:=P_{r_0}(z)-x\in\ker(\Delta^\ast_{J}+S_{r_0})$ and obtain $F_{r_0}(z-w)=((\Delta^\ast_{J}+S_{r_0})(z-w),P_{r_0}(z-w))=(y,x)$. Consequently, $F_{r_0}$ is bijective, and so $F_r$ is an isomorphism for all $r$ that are sufficiently close to $r_0$.\\
It is not very hard to show that $F^{-1}_r\circ F_{r_0}:D_{max}\rightarrow D_{max}$ maps $\ker(\Delta^\ast_{J}+S_{r_0})$ to $\ker(\Delta^\ast_{J}+S_{r})$, and moreover the space $V:=D_{min}+\ker(\Delta^\ast_{J}+S_{r_0})\subset D_{max}$ is closed. We define a family of maps 

\[\psi_r:V\oplus V^\perp\rightarrow D_{max},\quad (x+s)+y\mapsto x+(F^{-1}_r\circ F_{r_0})(s)+y,\]
where we use that $D_{min}\cap\ker(\Delta^\ast_{J}+S_{r_0})=\{0\}$ by \eqref{ucp}. Since $\psi_{r_0}=I_{D_{max}}$, we see that $\psi_r$ is an isomorphism for all $r$ that are sufficiently close to $r_0$, and moreover, $\psi_r(D_{min})=D_{min}$. Consequently, $\psi_r$ descends to a family of isomorphisms $\widetilde{\psi}_r:\beta\rightarrow\beta$ such that $\widetilde{\psi}_r(\ell(r_0))=\ell(r)$. We now let $\widetilde{P}_{r_0}$ denote the orthogonal projection onto $\ell(r_0)$ in $\beta$. Then $P_r=\widetilde{\psi}_r\widetilde{P}_{r_0}\widetilde{\psi}^{-1}_r:\beta\rightarrow\beta$ is a smooth family of projections such that $\im P_r=\ell(r)$. Finally, we set

\[P_{ort,r}:=P_rP^\ast_r(P_rP^\ast_r+(I_\beta-P^\ast_r)(I_\beta-P_r))^{-1},\]
which is by \cite[Lem. 12.8 a)]{BoossBuch} a smooth family of orthogonal projections in $\beta$ such that $\im(P_{ort,r})=\im P_r=\ell(r)$. Now the assertion follows from Lemma \ref{regcurve}.  
\end{proof}

It follows from \eqref{ucp} that $\ell(r)\cap\mu\neq \{0\}$ if and only if $r$ is a conjugate radius. Consequently, since $0$ is not conjugate by Lemma \ref{lem:Fredholm}, the Maslov index $\mu_{Mas}(\ell,\mu,[0,1])$ is defined whenever $1$ is not a conjugate radius. Note that, roughly speaking, $\mu_{Mas}(\ell, \mu, [0,1])$ counts radii for which the boundary value problems \eqref{bvplin} have non-trivial solutions (cf. App. \ref{App:Maslov}).


\section{Main theorems}
In this section we first state the main theorems of this paper, and afterwards we deduce as corollaries the Morse index theorem from Riemannian geometry and its generalisation to semi-Riemannian manifolds. For this latter part, we also recall the necessary definitions and constructions.

\begin{theorem}\label{mainI}
We assume that the boundary value problem \eqref{bvpnonlin} is non-degenerate, i.e., $r=1$ is not a conjugate radius for \eqref{bvplin}. Then

\[\sfl(h,[0,1])=\mu_{Mas}(\ell, \mu, [0,1])\in\mathbb{Z}.\]
\end{theorem}

Our second theorem treats the case in which $J=-I_k$, so that \eqref{bvplin} is strongly elliptic. Let us introduce the notation  

\begin{align}\label{mr}
m(r)=\dim\{u\in C^2(\Omega_r,\mathbb{R}^k):\,u\,\text{solves}\,\,\eqref{bvplin}\}.
\end{align}
The following result gives a new proof of Smale's theorem \cite{Smale} for the equations \eqref{bvplin}, and it generalises corresponding approaches from \cite{AleIchDomain} and \cite{AleIchBall} to systems.

\begin{theorem}\label{mainII}
If \eqref{bvpnonlin} is non-degenerate and $J=-I_k$, then there are only finitely many conjugate instants in $(0,1)$ and

\[\sfl(h,[0,1])=-\mu_{Morse}(h_1)=-\sum_{r\in(0,1)}{m(r)}.\]
\end{theorem}

For stating two corollaries of Theorem \ref{mainI} and Theorem \ref{mainII}, we first briefly recall some definitions and constructions for geodesics in semi-Riemannian manifolds, for which we refer to \cite{MPP} for a more detailed exposition. Let $(M,g)$ be a semi-Riemannian manifold of dimension $k$ and index $\nu$ having Levi Civita connection $\nabla$ and curvature $R$. If we fix two points $p,q\in M$, then the set $H^1_{pq}([0,1],M)$ of all curves of regularity $H^1$ that join $p$ and $q$ is a Hilbert manifold (cf. \cite{Klingenberg}). The functional

\[\mathcal{E}:H^1_{pq}([0,1],M)\rightarrow\mathbb{R},\quad \gamma\mapsto\int^1_0{g_{\gamma(t)}(\gamma'(t),\gamma'(t))\, dt}\]
is smooth and its critical points are the geodesics joining $p$ and $q$, i.e., the curves $\gamma:[0,1]\rightarrow M$ from $p$ to $q$ that satisfy the differential equation $\frac{\nabla}{dt}\gamma'=0$.\\
Let us now fix a a critical point $\gamma\in H^1_{pq}([0,1],M)$ of $\mathcal{E}$. The tangent space $T_{\gamma}H^1_{pq}([0,1],M)$ can be identified canonically with the set of all vector fields along $\gamma$ which are of regularity $H^1$ and vanish at their endpoints. The Hessian of the functional $\mathcal{E}$ at the critical point $\gamma$ is the bilinear form $D^2_\gamma\mathcal{E}:T_{\gamma}H^1_{pq}([0,1],M)\times T_{\gamma}H^1_{pq}([0,1],M)\rightarrow\mathbb{R}$ given by

\[D^2_\gamma\mathcal{E}(\xi,\eta)=\int^1_0{g\left(\frac{\nabla}{dx}\xi(x),\frac{\nabla}{dx}\eta(x)\right)\,dx}+\int^1_0{g(R(\gamma'(x),\xi(x))\gamma'(x),\eta(x))\, dx},\]
and the geodesic $\gamma$ is called \textit{non-degenerate} if $D^2_\gamma\mathcal{E}$ is non-degenerate.\\
If we now choose a parallel orthonormal frame $\{e^1,\ldots,e^k\}$ along $\gamma$, we can identify vector fields $\xi$ in $T_\gamma H^1_{pq}([0,1],M)$ with maps $u$ in $H^1_0([0,1],\mathbb{R}^k)$ by

\begin{align}\label{baserep}
\xi(x)=\sum^k_{i=1}{u_i(x)\,e^i(x)}, \quad x\in [0,1].
\end{align}
Under this identification, the quadratic form induced by $D^2_\gamma\mathcal{E}$ transforms to 

\[h(u)=-\int^1_0{\langle Ju'(x),u'(x)\rangle\,dx}+\int^1_0{\langle S(x)u(x),u(x)\rangle\,dx}\]
on $H^1_0([0,1],\mathbb{R}^k)$, where $S$ denotes the smooth path of symmetric matrices having components

\[S_{ij}(x)=g(R(\gamma'(x),e^i(x))\gamma'(x),e^j(x)),\quad x\in [0,1],\,\,1\leq i,j\leq k,\]
and 

\[J=\diag(\underbrace{-1,\ldots,-1}_{k-\nu},\underbrace{1,\ldots,1}_\nu)\]
as in \eqref{J}. If $(M,g)$ is a Riemannian manifold, then $J=-I_k$ and the finite number $\mu_{Morse}(h)$ is by definition the Morse index of the geodesic $\gamma$ (cf. \cite[\S 15]{MilnorMorse}).\\
If we now consider the restricted domain $\Omega_r=[0,r]$, we obtain as in Section \ref{section:genMorse} a family of quadratic forms by

\begin{align}\label{famh}
h_r(u)=-\int^1_0{\langle Ju'(x),u'(x)\rangle\,dx}+\int^1_0{\langle S_r(x)u(x),u(x)\rangle\,dx},\quad u\in H^1_0([0,1],\mathbb{R}^k),
\end{align}
where $S_r(x)=r^2S(r\cdot x)$. The kernel $\ker L_r$ of the associated Riesz representation \eqref{Rieszrep} of $h_r$ consists of all functions that satisfy the boundary value problem

\begin{equation}\label{bvp}
\left\{
\begin{aligned}
Ju''(x)+\,S_r(x)u(x)&=0,\,\,x\in [0,1],\\
u(0)=u(1)&=0.
\end{aligned}
\right.
\end{equation}
From \eqref{baserep}, it is readily seen that the space of all solutions of \eqref{bvp} is isomorphic to the space of all vector fields $\xi$ in $T_\gamma H^1_{pq}([0,1],M)$ that satisfy

\begin{align}\label{Jacobi}
\frac{\nabla^2}{dx^2}\xi(x)+R(\gamma'(x),\xi(x))\gamma'(x)=0
\end{align}
and vanish at $0$ and $r$. Equation \eqref{Jacobi} is called \textit{Jacobi equation}, and $r$ is a conjugate instant if

\[\mathscr M(r):=\dim\left\{\xi\in T_\gamma H^1_{pq}(I,M):\, \frac{\nabla^2}{dx^2}\xi(x)+R(\gamma'(x),\xi(x))\gamma'(x)=0,\quad \xi(0)=\xi(r)=0\right\}>0.\]
Since we just have seen that $\mathscr M(r)$ coincides with the dimension $m(r)$ of the space of solutions of \eqref{bvp}, we immediately obtain from Theorem \ref{mainII} the following corollary, which is the well known \textit{Morse index theorem in Riemannian geometry} (cf. \cite[\S 15]{MilnorMorse}).

\begin{cor}\label{corMorse}
If $\gamma$ is a non-degenerate geodesic in a Riemannian manifold $(M,g)$, then

\[\mu_{Morse}(\gamma)=\sum_{r\in (0,1)}{\mathscr M(r)}.\]
\end{cor} 

If $J\neq-I_k$, i.e., $(M,g)$ is not Riemannian, then $\mu_{Morse}(h)$ is infinite. Moreover, it was exposed by Helfer in \cite{Helfer} that conjugate points may accumulate in this case and so the indices in Corollary \ref{corMorse} are not defined in general. As observed in \cite{MPP}, a suitable generalisation of $\mu_{Morse}(\gamma)$ in the semi-Riemannian case is the spectral flow of the family \eqref{famh}. A possible way to overcome the problem of counting conjugate points along the geodesic is by using the Maslov index for curves of Lagrangian subspaces in $\mathbb{R}^{2k}$ as follows: we set $v=Ju'$ and see that the differential equations in \eqref{bvp} transform to the linear Hamiltonian systems

\begin{align}\label{Hamiltonian}
\begin{pmatrix}
u'\\
v'	
\end{pmatrix}=\sigma\begin{pmatrix}
-S_r(x)&0\\
0&-J	
\end{pmatrix}\begin{pmatrix}
	u\\
	v
\end{pmatrix},
\end{align} 
where 

\[\sigma=\begin{pmatrix}
	0&-I_k\\
	I_k&0
\end{pmatrix}\]
is the standard symplectic matrix. If $\Psi_r$ denotes the fundamental solution of \eqref{Hamiltonian}, that is, the unique matrix-valued solution that satisfies $\Psi_r(0)=I_{2k}$, then $\Psi_r(x)$ is symplectic for all $(r,x)\in[0,1]\times[0,1]$. Moreover, it follows immediately from the definition that $\Psi_r(1)(\{0\}\times\mathbb{R}^k)\cap(\{0\}\times\mathbb{R}^k)\neq\{0\}$ if and only if the boundary value problem \eqref{bvp} has a non-trivial solution, which means that $r$ is a conjugate instant. The \textit{Maslov index} $\mu_{Mas}(\gamma)$ of the geodesic $\gamma$ is defined as the Maslov index $\mu_{Mas}(\Psi_\cdot(1)(\{0\}\times\mathbb{R}^k),\{0\}\times\mathbb{R}^k,[0,1])$, where $\Psi_\cdot(1)(\{0\}\times\mathbb{R}^k)$ denotes the path $\ell(r)=\Psi_r(1)(\{0\}\times\mathbb{R}^k)\in\Lambda(\mathbb{R}^{2k})$ (cf. App. \ref{App:Maslov}). We prove below the following corollary of Theorem \ref{mainI}, which is the \textit{Morse index theorem for geodesics in semi-Riemannian manifolds} \cite[Prop. 6.1]{MPP} and a generalisation of the previous Corollary \ref{corMorse}.

\begin{cor}\label{maincor}
If $\gamma$ is a non-degenerate geodesic in a semi-Riemannian manifold $(M,g)$, then

\[\mu_{Morse}(\gamma)=\mu_{Mas}(\gamma).\]
\end{cor}


\section{Proofs}
In this section we prove Theorem \ref{mainI}, Theorem \ref{mainII} and Corollary \ref{maincor}. Note that Corollary \ref{corMorse} is an immediate consequence of Theorem \ref{mainII} and hence does not need to be proved.


\subsection{Proof of Theorem \ref{mainI}}\label{sect:proofI}
We begin by recalling that $h_0$ is non-degenerate by Lemma \ref{lem:Fredholm}. Consequently, there exists $r^\ast>0$ such that $h_r$ is non-degenerate for all $r\in[0,r^\ast]$ and properties ii) and iii) in Section \ref{sect-sfl} show that $\sfl(h,[0,1])=\sfl(h,[r^\ast,1])$. On the other side, we deduce from Lemma \ref{lemmaconj} that $\ell(r)\cap\mu=\{0\}$ for all $r\in[0,r]$, and hence $\mu_{Mas}(\ell,\mu,[0,1])=\mu_{Mas}(\ell,\mu,[r^\ast,1])$ by i) and ii) in Appendix \ref{App:Maslov}. So in what follows, we may restrict to the interval $[r^\ast,1]$.\\ 
We define for $r\in[r^\ast,1]$ quadratic forms $h^\delta_r:H^1_0(\Omega,\mathbb{R}^k)\rightarrow\mathbb{R}$ by

\begin{align*}
h^\delta_r(u)&=\frac{1}{r^2}h_r(u)+\delta\|u\|^2_{L^2(\Omega,\mathbb{R}^k)}\\
&=-\frac{1}{r^2}\sum^k_{j=1}{(-1)^{\nu(j)}\int_\Omega{\langle \nabla u^j,\nabla u^j\rangle\,dx}}+\int_\Omega{\langle S(r\cdot x)u(x),u(x)\rangle\,dx}+\delta\int_\Omega{\langle u(x),u(x)\rangle\,dx},
\end{align*}
and we consider the unbounded selfadjoint Fredholm operators

\[\mathcal{A}^\delta_r:L^2(\Omega,\mathbb{R}^k)\supset D\rightarrow L^2(\Omega,\mathbb{R}^k),\quad \mathcal{A}^\delta_ru=\frac{1}{r^2}\,J\Delta u(x)+S(r\cdot x)u(x)+\delta\,u(x)\]
on the domain $D=H^2(\Omega,\mathbb{R}^k)\cap H^1_0(\Omega,\mathbb{R}^k)$. From integration by parts, we see that $h^\delta_r$ is non-degenerate if and only if $\mathcal{A}^\delta_r$ is invertible, and moreover

\begin{align}\label{crossformh}
h^\delta_r(u)=\langle\mathcal{A}^\delta_ru,u\rangle_{L^2(\Omega,\mathbb{R}^k)},\quad u\in D,\,\, (r,\delta)\in[r^\ast,1]\times\mathbb{R},
\end{align}
Since $h_{r^\ast}$ and $h_1$ are non-degenerate, there is $\delta^\ast>0$ such that $h^\delta_{r^\ast}$ and $h^\delta_1$ are non-degenerate for all$|\delta|<\delta^\ast$. It follows from property i) of the spectral flow in Section \ref{sect-sfl} that 

\begin{align}\label{hhast}
\sfl(h,[r^\ast,1])=\sfl(h^\delta,[r^\ast,1]),\quad\text{for all}\, \delta\in[-\delta^\ast,\delta^\ast].
\end{align}
Since $\mathcal{A}^\delta_r$ is invertible if $h^\delta_r$ is non-degenerate, we see that $\mathcal{A}^\delta_{r^\ast}$ and $\mathcal{A}^\delta_1$ are invertible for all $\delta\in[-\delta^\ast,\delta^\ast]$, and consequently, $\sfl(\mathcal{A}^\delta,[r^\ast,1])$ is defined for all these $\delta$. By Theorem \ref{thm:Sard} there exists $\delta\in[-\delta^\ast,\delta^\ast]$ such that $\mathcal{A}^\delta$ has only regular crossings in $[r^\ast,1]$. Since the crossing forms of $h^\delta$ and $\mathcal{A}^\delta$ coincide by \eqref{crossformh}, we see by \eqref{hhast} and Theorem \ref{thm:crossing} that

\begin{align}\label{crossformA}
\sfl(h,[r^\ast,1])=\sum_{r\in(r^\ast,1)}{\sgn\Gamma(\mathcal{A}^\delta,r)}.
\end{align} 
Let us now consider a regular crossing $r_0\in (r^\ast,1)$. The crossing form is by definition

\begin{align*}
\begin{split}
\Gamma(\mathcal{A}^\delta,r_0):\ker\mathcal{A}^\delta_{r_0}\rightarrow\mathbb{R},\quad \Gamma(\mathcal{A},r_0)[u]&=-\frac{2}{r^3_0}\int_{\Omega}{\langle J(\Delta u)(x),u(x)\rangle\, dx}\\
&+\frac{d}{dr}\mid_{r=r_0}\int_\Omega{\left\langle S(r\cdot x)u(x),u(x)\right\rangle\,dx}.
\end{split}
\end{align*}
Let $u=(u^1,\ldots,u^k):\Omega\rightarrow\mathbb{R}^k$ be an element of $\ker\mathcal{A}^\delta_{r_0}$, i.e.,

\begin{align}\label{equ}
\frac{1}{r^2_0}J(\Delta u)(x)+S(r_0\cdot x)u(x)+\delta u(x)=0,\quad x\in \Omega.
\end{align}
We define for $r\in(0,r_0]$ functions $u_r:\Omega\rightarrow\mathbb{R}^k$ by 

\[u_r(x)=u\left(\frac{r}{r_0}x\right).\]
We note that $u_{r_0}=u$, and

\begin{align}\label{equII}
\begin{split}
&\frac{1}{r^2}(J\Delta u_r)(x)+S(r\cdot x)u_r(x)+\delta u_r(x)=\frac{1}{r^2_0}J\Delta u\left(\frac{r}{r_0}x\right)+S(r\cdot x)u\left(\frac{r}{r_0}x\right)+\delta u_r(x)\\
&=\frac{1}{r^2_0}J\Delta u\left(\frac{r}{r_0}x\right)+S(r_0\frac{r}{r_0}\cdot x)u\left(\frac{r}{r_0}x\right)+\delta u\left(\frac{r}{r_0}x\right)=0,\quad x\in\Omega.
\end{split}
\end{align}
Let us set for notational convenience

\begin{align}\label{dotu}
\dot{u}(x):=\frac{d}{dr}\mid_{r=r_0}u_r(x)=\frac{1}{r_0}(D_xu)x=\frac{1}{r_0}(\langle\nabla u^1(x),x\rangle,\ldots,\langle\nabla u^k(x),x\rangle)^T\in\mathbb{R}^k,
\end{align}
where $\cdot{}^T$ denotes the transpose. If we differentiate \eqref{equII} by $r$ and evaluate at $r=r_0$, we obtain

\[-\frac{2}{r^3_0}J\Delta u(x)+\frac{1}{r^2_0}J\Delta \dot{u}(x)+\frac{d}{dr}\mid_{r=r_0}(S(r\cdot x))u(x)+S(r_0\,x)\dot{u}(x)+\delta \dot{u}(x)=0,\quad x\in\Omega.\]
We take scalar products with $u$, integrate over $\Omega$ and see that

\begin{align*}
&-\frac{2}{r^3_0}\int_\Omega{\langle J\Delta u(x),u(x)\rangle\,dx}+\frac{1}{r^2_0} \int_\Omega{\langle J\Delta \dot{u}(x),u(x)\rangle\, dx}\\
&+\int_\Omega{\frac{d}{dr}\mid_{r=r_0}\langle S(r\cdot x)u(x),u(x)\rangle\,dx}+\int_\Omega{\langle S(r_0\cdot x)\dot{u}(x),u(x)\rangle\, dx}+\delta \int_\Omega{\langle\dot{u}(x),u(x)\rangle\,dx}=0.
\end{align*}
Consequently,

\[\Gamma(\mathcal{A}^\delta,r_0)[u]=-\frac{1}{r^2_0}\int_\Omega{\langle J\Delta \dot{u}(x),u(x)\rangle\, dx}-\int_\Omega{\langle S(r_0\cdot x)\dot{u}(x),u(x)\rangle\, dx}-\delta \int_\Omega{\langle\dot{u}(x),u(x)\rangle\,dx}\]
and a subsequent integration by parts gives

\begin{align*}
\Gamma(\mathcal{A}^\delta,r_0)[u]&=-\frac{1}{r^2_0}\int_\Omega{\langle  \dot{u}(x),J\Delta u(x)\rangle\, dx}-\frac{1}{r^2_0}\int_{\partial\Omega}{\langle J\partial_n\dot{u}(x),u(x)\rangle\,dS}\\
&+\frac{1}{r^2_0}\int_{\partial\Omega}{\langle J\dot{u}(x),\partial_nu(x)\rangle\,dS}-\int_\Omega{\langle S(r_0\cdot x)\dot{u}(x),u(x)\rangle\, dx}-\delta \int_\Omega{\langle\dot{u}(x),u(x)\rangle\,dx},
\end{align*}
where we denote $\partial_nu(x)=(\partial_nu^1(x),\ldots,\partial_nu^k(x))^T$ and as before $\partial_nu^j(x)=\sum^n_{l=1}{\frac{\partial u^j}{\partial x^l}(x)\nu^l(x)}$.
Since $u$ solves the boundary value problem \eqref{bvplin}, it follows that

\begin{align}\label{crossform}
\Gamma(\mathcal{A}^\delta,r_0)[u]=\frac{1}{r^2_0}\int_{\partial\Omega}{\langle J\dot{u}(x),\partial_nu(x)\rangle\,dS}.
\end{align}
We now consider the Maslov index $\mu_{Mas}(\ell,\mu,[r^\ast,1])$, and we introduce a curve $\ell^\delta:[r^\ast,1]\rightarrow\mathcal{FL}_\mu(\beta)$ by 

\[\ell^\delta(r)=\tau\left(\{u\in D_{max}:J\Delta u(x)+S_r(x)u(x)+r^2\delta\, u(x)=0\}\right),\]
where $-\delta^\ast<\delta<\delta^\ast$ is chosen as above. Let us recall that $\tau:D_{max}\rightarrow\beta=D_{max}/D_{min}$ denotes the canonical projection. It is readily seen from \eqref{ucp} that the linear maps 

\begin{align}\label{gammaiso}
\tau\mid_{\ker\mathcal{A}^\delta_r}:\ker\mathcal{A}^\delta_r\rightarrow \ell^\delta(r)\cap\mu
\end{align}
are isomorphisms for all $r\in[r^\ast,1]$.\\
We consider the homotopy $h:[r^\ast,1]\times [0,1]\rightarrow\mathcal{FL}_\mu(\beta)$, 

\[h(r,s)=\tau\left(\{u\in D_{max}:J\Delta u(x)+S_r(x)u(x)+s\cdot r^2\delta u(x)=0\}\right),\]
which is continuous by Proposition \ref{contell}. The isomorphisms \eqref{gammaiso} show that  

\[h(r^\ast,s)\cap\mu=h(1,s)\cap\mu=\{0\}\quad\text{for all}\, s\in [0,1],\]
and so it follows from property iii) in Appendix \ref{App:Maslov} that 

\begin{align}\label{masdelta}
\mu_{Mas}(\ell,\mu,[r^\ast,1])=\mu_{Mas}(\ell^\delta,\mu,[r^\ast,1]).
\end{align}
Since $\ker\mathcal{A}^\delta_r\neq\{0\}$ if and only if $\ell^\delta(r)\cap\mu\neq0$ by \eqref{gammaiso}, we can henceforth assume that $r_0$ is the only crossing of $\ell^\delta$ in $[r^\ast,1]$. The task is now to compute the corresponding crossing form $\Gamma(\ell^\delta,\mu;r_0)$. Let $y\in\ell^\delta(r_0)\cap\mu$. By \eqref{gammaiso} we can take 

\[u\in\ker\mathcal{A}^\delta_{r_0}=\{v\in D:\,J\Delta v(x)+S_{r_0}(x)v(x)+r^2_0\delta v(x)=0\}\]
such that $\tau(u)=y$. As in \eqref{equII}, we see that 

\[J\Delta u_r(x)+S_r(x)u_r(x)+r^2\delta\, u_r(x)=0,\quad x\in\Omega,\]
when $u_r(x)=u\left(\frac{r}{r_0}x\right)$ is defined for $r$ sufficiently close to $r_0$. Consequently, $X(r):=\tau(u_r)\in\beta$ is in $\ell^\delta(r)$, and moreover, $X$ depends smoothly on $r$ since $u$ is smooth by standard regularity theory.\\
Let now $\varphi_r:\ell^\delta(r_0)\rightarrow\ell^\delta(r_0)^\perp$ be a family of maps such that $\gra\varphi_r=\ell^\delta(r)$ for $|r-r_0|$ sufficiently small (cf. App. \ref{App:Maslov}). We define $c(r)=P(X(r))$, where $P:\beta\rightarrow\beta$ denotes the orthogonal projection onto $\ell^\delta(r_0)$, and obtain a smooth curve $c$ in $\ell^\delta(r_0)$ such that 

\[X(r)=c(r)+\varphi_r(c(r))\in\ell^\delta(r).\]
Note that $c(r_0)=X(r_0)=\tau(u)=y$ since $u_{r_0}=u$. Moreover, $\varphi_{r_0}\equiv0$ and so $\dot{c}(r_0)+\varphi_{r_0}(\dot{c}(r_0))=\dot{c}(r_0)\in\ell^\delta(r_0)$. It follows that

\begin{align*}
\omega(X(r_0),\dot{X}(r_0))&=\omega(y,\dot{c}(r_0)+\varphi_{r_0}(\dot{c}(r_0))+\dot{\varphi}_{r_0}(c(r_0)))\\
&=\omega(y,\dot{c}(r_0)+\varphi_{r_0}(\dot{c}(r_0)))+\omega(y,\dot{\varphi}_{r_0}(c(r_0)))\\
&=\omega(y,\dot{\varphi}_{r_0}(c(r_0)))=\omega(y,\dot{\varphi}_{r_0}(y)),
\end{align*}
and we have

\begin{align*}
\Gamma(\ell^\delta,\mu;r_0)[y]&=\frac{d}{dr}\mid_{r=r_0}\omega(y,\varphi_r(y))=\omega(X(r_0),\dot{X}(r_0))=\omega(\tau(u),\tau(\dot u)).
\end{align*}
Since $u$ is smooth and vanishes on $\partial\Omega$, we finally see from \eqref{omegaregular} that

\begin{align*}
\Gamma(\ell^\delta,\mu;r_0)[y]&=-\int_{\partial\Omega}{\langle J\partial_n\dot{u}(x),u(x)\rangle\, dS}+\int_{\partial\Omega}{\langle J\dot{u}(x),\partial_nu(x)\rangle\, dS}\\
&=\int_{\partial\Omega}{\langle J\dot{u}(x),\partial_nu(x)\rangle\, dS}.
\end{align*}
By \eqref{crossform}, we thus have shown that

\[\Gamma(\ell^\delta,\mu;r_0)[\tau(u)]=r^2_0\,\Gamma(\mathcal{A}^\delta,r_0)[u],\quad u\in\ker\mathcal{A}^\delta_{r_0}.\]
Since \eqref{gammaiso} is an isomorphism, it follows that $\Gamma(\ell^\delta,\mu;r_0)$ is non-degenerate and so we finally conclude from Proposition \ref{thm:localcontribution}

\[\mu_{Mas}(\ell^\delta,\mu,[r^\ast,1])=\sgn\Gamma(\ell^\delta,\mu;r_0)=\sgn \Gamma(\mathcal{A}^\delta,r_0),\]
which proves Theorem \ref{mainI}.


\subsection{Proof of Theorem \ref{mainII}}
Let us go back to the beginning of our proof of Theorem \ref{mainI} and let us now set $\delta=0$. As before, we see that the crossings and crossing forms of the operators $\mathcal{A}^0$ and $h^0$ coincide (cf. \eqref{crossformh}). Let $r_0\in(r^\ast,1)$ be a crossing of $\mathcal{A}^0$. By \eqref{crossform} we have

\[\Gamma(\mathcal{A}^0,r_0)[u]=-\frac{1}{r^2_0}\int_{\partial\Omega}{\langle \dot{u}(x),\partial_nu(x)\rangle\,dS},
\]
and since $\dot{u}(x)=\frac{1}{r_0}(D_xu)x$ by \eqref{dotu}, we obtain

\begin{align*}
\Gamma(\mathcal{A}^0,r_0)[u]=-\frac{1}{r^3_0}\sum^k_{j=1}{\int_{\partial\Omega}{\langle\nabla u^j(x),x\rangle\,\langle\nabla u^j(x),\nu(x)\rangle}\,dS}.
\end{align*}
If we denote by $x^t$ the component of $x$ tangential to the boundary $\partial\Omega$, then 

\[x=\langle x,\nu(x)\rangle\nu(x)+x^t,\]
and so

\[\langle\nabla u^j(x),x\rangle=\langle\nabla u^j(x),\nu(x)\rangle\,\langle x,\nu(x)\rangle+\langle\nabla u^j(x),x^t\rangle.\]
We obtain

\begin{align*}
\Gamma(\mathcal{A}^0,r_0)[u]&=-\frac{1}{r^3_0}\sum^k_{j=1}{\int_{\partial\Omega}{(\partial_nu^j(x))^2\,\langle x,\nu(x)\rangle}\,dS}\\
&-\frac{1}{r^3_0}\sum^k_{j=1}{\int_{\partial\Omega}{\partial_nu^j(x)\,\langle\nabla u^j(x),x^t\rangle}\,dS}.
\end{align*}
Since

\begin{align*}
&\diverg(u^j(x)\,(\partial_nu^j)(x)\, x^t)=(\partial_nu^j)(x)\langle\nabla u^j(x),x^t\rangle+u^j(x)\langle\nabla(\partial_nu^j)(x),x^t\rangle\\
&+u^j(x)(\partial_nu^j)(x)\diverg(x^t)
\end{align*}
and $u^j\mid_{\partial\Omega}=0$, we see that 

\[(\partial_nu^j)(x)\,\langle\nabla u^j(x),x^t \rangle=\diverg(u^j(x)(\partial_nu^j)(x) x^t),\quad x\in\partial \Omega,\]
and now Stokes' theorem gives

\begin{align}\label{inequ}
\Gamma(\mathcal{A}^0,r_0)[u]=-\frac{1}{r^3_0}\sum^k_{j=1}{\int_{\partial\Omega}{((\partial_nu^j)(x))^2\,\langle x,\nu(x)\rangle}\,dS}\leq 0,
\end{align}
where we use that $\langle x,\nu(x)\rangle>0$, $x\in\partial\Omega$, since $\Omega$ is star-shaped with respect to $0$. Finally, even the strict inequality holds in \eqref{inequ}, for otherwise $\partial_nu=0$ which implies that $u\equiv 0$ by \eqref{ucp}.\\
Consequently, $\Gamma(\mathcal{A}^0,r_0)$ is negative definite, and so in particular non-degenerate. Moreover, since the crossing forms of $\mathcal{A}^0$ and $h^0$ coincide by \eqref{crossformh}, $h^0$ has only regular crossings as well. Finally we conclude from Lemma \ref{Morse=sfl}, Theorem \ref{thm:crossing} and our choice of $r^\ast$ in Section \ref{sect:proofI} that 

\begin{align}\label{contributionreg}
\begin{split}
\mu_{Morse}(h_1)&=-\sfl(h,[0,1])=-\sfl(h^0,[r^\ast,1])=-\sfl(\mathcal{A}^0,[r^\ast,1])\\
&=-\sum_{r\in (r^\ast,1)}{\sgn\Gamma(\mathcal{A}^0,r)}=\sum_{r\in(r^\ast,1)}{\dim\ker\mathcal{A}^0_r}=\sum_{r\in(0,1)}{\dim\ker\mathcal{A}^0_r},
\end{split}
\end{align}
where we have used that $\Gamma(\mathcal{A}^0,r)$ is negative definite. Now the assertion follows from $m(r)=\dim\ker\mathcal{A}^0_r$ (cf. \eqref{mr}).


\subsection{Proof of Corollary \ref{maincor}}
Let us first recall from \cite[\S 6.4]{Weidmann} that for $\Omega=(0,1)$, the space $D_{max}$ is just $H^2([0,1],\mathbb{R}^k)$. We define a map

\[\varphi:\beta\rightarrow\mathbb{R}^{2k}\oplus\mathbb{R}^{2k},\quad \varphi(\tau(u))\mapsto ((u(0),Ju'(0)),(u(1),Ju'(1))).\]
and note that $\varphi$ is well defined and injective since $(u(0),Ju'(0),u(1),Ju'(1))=0$ for all $u\in H^2_0([0,1],\mathbb{R}^k)$. Moreover, it is clear that every element in $\mathbb{R}^{2k}\oplus\mathbb{R}^{2k}$ can be obtained as image under $\varphi$ of some element in $H^2([0,1],\mathbb{R}^k)$, and consequently, $\varphi$ is an isomorphism. Finally, we obtain from \eqref{omegaregular}

\begin{align*}
\omega(\tau(u),\tau(v))&=\langle J u'(1),v(1)\rangle-\langle Ju'(0),v(0)\rangle-\langle u(1),Jv'(1)\rangle+\langle u(0),Jv'(0)\rangle\\
&=\left\langle\begin{pmatrix}
0&I_k\\
-I_k&0	
\end{pmatrix}\begin{pmatrix}
u(1)\\
Ju'(1)	
\end{pmatrix},\begin{pmatrix}
	v(1)\\
	Jv'(1)
\end{pmatrix}\right\rangle-\left\langle\begin{pmatrix}
0&I_k\\
-I_k&0	
\end{pmatrix}\begin{pmatrix}
u(0)\\
Ju'(0)	
\end{pmatrix},\begin{pmatrix}
	v(0)\\
	Jv'(0)
\end{pmatrix}\right\rangle\\
&=-\omega_2(\varphi([u]),\varphi([v]),
\end{align*}
where $\omega_2=\omega_0\times(-\omega_0)$ was defined in Appendix \ref{App:Maslov}. Let us recall that by definition

\[\ell(r)=\tau(\{u\in H^2([0,1],\mathbb{R}^k):\, Ju''(x)+S_r(x)u(x)=0,\quad x\in [0,1]\}),\quad r\in[0,1].\]
We set 

\[\widetilde{\ell}(r):=\varphi(\ell(r))=\{((u(0),Ju'(0)),(u(1),Ju'(1)):\, Ju''(x)+S_r(x)u(x)=0,\, x\in [0,1]\}\]
for $r\in[0,1]$, $\widetilde{\mu}:=\varphi(\mu)=\{0\}\times\mathbb{R}^k\times\{0\}\times\mathbb{R}^k$, and conclude that

\begin{align}\label{maslovequ}
\mu_{Mas}(\ell,\eta,[0,1])=\mu_{Mas}(\widetilde{\ell},\widetilde{\mu},[0,1]).
\end{align}
Let us now consider as in the definition of $\mu_{Mas}(\gamma)$ the fundamental solutions $\Psi_r$ of the differential equations \eqref{Hamiltonian}, and let us write

\begin{align*}
\Psi_r(x)=\begin{pmatrix}
a_r(x)&b_r(x)\\
c_r(x)&d_r(x)	
\end{pmatrix},\quad x\in [0,1],
\end{align*}
for some $k\times k$-matrices $a_r,b_r,c_r$ and $d_r$, $r\in[0,1]$. We obtain from \eqref{Hamiltonian}

\[\begin{pmatrix}
	a'_r(x)&b'_r(x)\\
	c'_r(x)&d'_r(x)
\end{pmatrix}=\begin{pmatrix}
	Jc_r(x)&Jd_r(x)\\
	-S_r(x)a_r(x)&-S_r(x)b_r(x)
\end{pmatrix}\]
and we see that the general solution of the differential equation $Ju''(x)+S_r(x)u(x)=0$ is given by $u(x)=a_r(x)u(0)+b_r(x)Ju'(0)$, $x\in[0,1]$. From this and $Ju'(x)=c_r(x)u(0)+d_r(x)Ju'(0)$, we conclude that

\[\begin{pmatrix}
u(x)\\
Ju'(x)	
\end{pmatrix}=\begin{pmatrix}
a_r(x)&b_r(x)\\
c_r(x)&d_r(x)	
\end{pmatrix}\begin{pmatrix}
	u(0)\\
	Ju'(0)
\end{pmatrix},\quad x\in[0,1],\]
and consequently

\[\widetilde{\ell}(r)=\{(w,\Psi_r(1)w):\, w\in\mathbb{R}^{2k}\}=\gra\Psi_r(1)\subset\mathbb{R}^{2k}\oplus\mathbb{R}^{2k}.\]
Setting $\mu=\mu'=\{0\}\times\mathbb{R}^k$ in \eqref{RobSalgra} we finally obtain from \eqref{maslovequ}

\begin{align*}
\mu_{Mas}(\ell,\mu,[0,1])&=\mu(\gra\Psi_\cdot(1),(\{0\}\times\mathbb{R}^k)\times(\{0\}\times\mathbb{R}^k),[0,1])\\
&=\mu_{Mas}(\Psi_\cdot(1)(\{0\}\times\mathbb{R}^k),\{0\}\times\mathbb{R}^k,[0,1])\\
&=\mu_{Mas}(\gamma).
\end{align*}


\section{Bifurcation}
In this section we use the bifurcation theory developed in \cite{FPR} and \cite{JacoboIch} to study bifurcation phenomena for solutions of semilinear elliptic partial differential equations under shrinking of the domain. Our results will improve the papers \cite{AleIchDomain} and \cite{AleIchBall} of the authors, which were discussed in detail in the second author's survey \cite{WaterstraatSurvey}.\\
Let $H$ be a separable Hilbert space and $f:[a,b]\times H\rightarrow\mathbb{R}$ a continuous function such that each $f_\lambda:=f(\lambda,\cdot):H\rightarrow\mathbb{R}$ is $C^2$ and its first and second derivatives depend continuously on $\lambda\in[a,b]$. In what follows, we assume that $0\in H$ is a critical point of all $f_\lambda$, $\lambda\in[a,b]$. 

\begin{defi}\label{defbif}
We call $\lambda^\ast\in[a,b]$ a bifurcation point of critical points of $f$ if every neighbourhood of $(\lambda^\ast,0)$ in $[a,b]\times H$ contains elements $(\lambda,u)$ such that $u\neq 0$ is a critical point of $f_\lambda$.
\end{defi}
Since the second derivatives $D^2_0f_\lambda$ at the critical point $0\in H$ are bounded symmetric bilinear forms, there exists a unique continuous path of bounded selfadjoint operators $L$, $\lambda\in [a,b]$ on $H$ (cf. \eqref{Rieszrep}) such that 

\begin{align}\label{Riesz}
D^2_0f_\lambda(u,v)=\langle L_\lambda u,v\rangle_H,\quad u,v\in H.
\end{align}
The following assertion is an immediate consequence of the well-known implicit function theorem in Banach spaces (cf. \cite[\S 2.2]{Ambrosetti}).

\begin{lemma}\label{implicitfct}
If $\lambda^\ast$ is a bifurcation point of critical points of $f$, then $L_{\lambda^\ast}$ is not invertible.
\end{lemma}
However, $\lambda^\ast$ need not to be a bifurcation point if $L_{\lambda_\ast}$ is non-invertible, i.e., the converse statement of Lemma \ref{implicitfct} is false in general.\\
Let us now assume that $L_\lambda$ is Fredholm for all $\lambda\in[a,b]$ and that $L_a,L_b$ are invertible, so that the spectral flow of the path $L:[a,b]\rightarrow\mathcal{FS}(H)$ of bounded selfadjoint Fredholm operators on $H$ is defined. A proof of the next theorem can be found in \cite{JacoboIch}.

\begin{theorem}\label{mainbif}   
If $\sfl(L,[a,b])\neq 0$, then there exists a bifurcation point $\lambda^\ast\in (a,b)$ of critical points of $f$ from the trivial branch.
\end{theorem}
In some situations there is an a priori bound on the dimension of the kernels of the operators $L_\lambda$. The following result shows that then the number of bifurcation points can be estimated from below (cf. \cite[Thm. 2.1 ii)]{JacoboIch}).

\begin{theorem}\label{biffinite}
Assume that there exist only finitely many $\lambda\in(a,b)$ such that $\ker L_\lambda\neq 0$. Then there are at least

\[\left\lfloor\frac{|\sfl L|}{\max_{\lambda\in I}\dim\ker L_\lambda}\right\rfloor\]
distinct bifurcation points of critical points from the trivial branch $[a,b]\times\{0\}$.
\end{theorem}

We now want to apply the previous bifurcation theorems in our setting. Let us recall that in the definition of the generalised Morse index in Section \ref{section:genMorse}, we have introduced a family of $C^2$-functionals $\psi_r: H^1_0(\Omega_r,\mathbb{R}^k)\rightarrow\mathbb{R}$, $r\in(0,1]$, such that the critical points of $\psi_r$ are precisely the weak solutions of the semilinear equation \eqref{bvpnonlin}.\\
Let us recall that $r^\ast\in(0,1]$ is a \textit{bifurcation radius} if there exist a sequence $\{r_n\}_{n\in\mathbb{N}}\subset(0,1]$ and weak solutions $0\neq u_n\in H^1_0(\Omega_{r_n},\mathbb{R}^k)$ of \eqref{bvpnonlin} such that $r_n\rightarrow r^\ast$ and $\|u_n\|_{H^1_0(\Omega_{r_n},\mathbb{R}^k)}\rightarrow 0$ for $n\rightarrow\infty$. Clearly, $r^\ast$ is a bifurcation radius for the semilinear equations \eqref{bvpnonlin} if and only if it is a bifurcation point in the sense of Definition \ref{defbif} for the family of functionals $\tilde{\psi}:[0,1]\times H^1_0(\Omega,\mathbb{R}^k)\rightarrow\mathbb{R}$ defined in \eqref{tildepsi}. Note that the quadratic forms $h_r$ from the definition of the generalised Morse index in Section \ref{section:genMorse} are induced by the Hessians of $\tilde{\psi}_r$ at $0\in H^1_0(\Omega,\mathbb{R}^k)$. Consequently, we obtain from our main Theorem \ref{mainI} the remarkable result that the existence of bifurcation radii can be deduced from the images of $\ker(\Delta^\ast_J+S_r)$ under $\tau$ in the symplectic Hilbert space $\beta$:

\begin{theorem}\label{mainbifI}
If the assumptions of Theorem \ref{mainI} hold and if $\mu_{mas}(\ell,\mu,[0,1])\neq 0$, then there exists a bifurcation radius $r^\ast\in(0,1)$.
\end{theorem}

In the proof of Theorem \ref{mainII} we showed that if $J=-I_k$, then there are only finitely many crossings, and at each crossing $r_0$ of $h$ the contribution to the spectral flow is the dimension of the solution space of \eqref{bvplin} (cf, \eqref{contributionreg}). Consequently, we obtain the following theorem, which extends the main theorems of \cite{AleIchDomain} and \cite{AleIchBall} to strongly elliptic systems.

\begin{theorem}\label{mainbifII}
If the assumptions of Theorem \ref{mainI} hold and $J=-I_k$, then the bifurcation radii of \eqref{bvpnonlin} are precisely the conjugate radii of \eqref{bvplin}.  
\end{theorem}  
Note that this means in particular that the converse of Lemma \ref{implicitfct} is true under the assumptions of Theorem \ref{mainbifII}.\\
Another interesting special case is $n=1$, i.e., systems of ordinary differential equations. Then the dimensions of the solution spaces of the boundary value problems \eqref{bvplin} can be estimated above by the space dimension $k$, and so we immediately obtain the following corollary

\begin{cor}
If $n=1$ in Theorem \ref{mainbifI} and if there are only finitely many radii $r\in (0,1)$ for which $h_r$ is degenerate, then there are at least 

\[\left\lfloor\frac{|\mu_{Mas}(\ell,\mu,[0,1])|}{k}\right\rfloor\]
distinct bifurcation radii.
\end{cor}  

Since conjugate radii are isolated for $J=-I_k$, we deduce from the previous corollary and Lemma \ref{Morse=sfl} the following result:

\begin{cor}
If $n=1$ and $J=-I_k$ in Theorem \ref{mainbifI}, then there are at least 

\[\left\lfloor\frac{\mu_{Morse}(h)}{k}\right\rfloor\]
distinct bifurcation radii.
\end{cor}

Finally, we want to point out the strength of our bifurcation theory by two examples:

\subsubsection*{Example I:\, Non variational perturbations}
Let us consider on $\Omega=\left[0,\frac{3}{2}\pi\right]$ the ordinary differential equations

\begin{equation}\label{ODEI}
\left\{
\begin{aligned}
-u''(x)-u(x)-u(x)^2v(x)^3&=0,\quad x\in\left[0,\frac{3}{2}\pi\right],\\
-v''(x)-v(x)+u(x)^3v(x)^2&=0\\
u(0)=v(0)=u\left(\frac{3}{2}\pi\right)=v\left(\frac{3}{2}\pi\right)&=0.
\end{aligned}
\right.
\end{equation}
If we multiply the first equation in \eqref{ODEI} by $v$, the second one by $u$ and subtract them, we obtain for all $0<r\leq \frac{3}{2}\pi$

\begin{align*}
0&=-\int ^r_0{(u^2v^4+u^4v^2)\,dx}-\int^r_0{(vu''-uv'')\,dx}=-\int ^r_0{(u^2v^4+u^4v^2)\,dx}\leq 0,
\end{align*}
and hence all solutions of the restricted equations

\begin{equation*}
\left\{
\begin{aligned}
-u''(x)-u(x)-u(x)^2v(x)^3&=0,\quad x\in[0,r],\\
-v''(x)-v(x)+u(x)^3v(x)^2&=0\\
u(0)=v(0)=u(r)=v(r)&=0.
\end{aligned}
\right.
\end{equation*}
are trivial. However, we see from the linearisation

\begin{equation*}
\left\{
\begin{aligned}
-u''(x)-u(x)&=0,\quad x\in[0,r],\\
-v''(x)-v(x)&=0\\
u(0)=v(0)=u(r)=v(r)&=0
\end{aligned}
\right.
\end{equation*}
that $r=\pi$ is a conjugate radius. Note that this is not a contradiction to Theorem \ref{mainbifII}, since 

\[V(u,v)=\begin{pmatrix}
	-u^2v^3\\
	+u^3v^2
\end{pmatrix}\]
is not a gradient vector field.

\subsubsection*{Example II: A conjugate radius which is not a bifurcation radius}
If we consider instead 

\begin{equation}\label{exampleequ}
\left\{
\begin{aligned}
u''(x)+u(x)+u(x)^2v(x)^3&=0,\quad x\in[0,r],\\
-v''(x)-v(x)+u(x)^3v(x)^2&=0\\
u(0)=v(0)=u(r)=v(r)&=0,
\end{aligned}
\right.
\end{equation} 
for $0<r\leq \frac{3}{2}\pi$, then of course, there are still no non-trivial solutions and $r=\pi$ is the only conjugate radius.\\
The corresponding quadratic forms $h_r$ in Section \ref{section:genMorse} are given by

\[h_r(u,v)=-\int^{\frac{3}{2}\pi}_0{(u'(x))^2dx}+\int^{\frac{3}{2}\pi}_0{(v'(x))^2dx}+r^2\int^{\frac{3}{2}\pi}_0{((u(x))^2-(v(x))^2)\,dx},\]
where $(u,v)\in H^1_0\left(\left[0,\frac{3}{2}\pi\right],\mathbb{R}^2\right)$, and the crossing form at $r=\pi$ is

\[\Gamma(h,\pi)[(u,v)]=2\pi\int^{\frac{3}{2}\pi}_0{((u(x))^2-(v(x))^2)\,dx},\quad (u,v)\in\ker L_\pi.\]
Since the kernel of $L_\pi$ is given by the solutions of the linearisation of \eqref{exampleequ} (cf. Lemma \ref{lemmaconj}), we see that

\[\ker L_\pi=\{(a\sin(\cdot),b\sin(\cdot)):\quad a,b\in\mathbb{R}\}.\]
Consequently, $\Gamma(h,\pi)$ is non-degenerate but $\sgn\Gamma(L,\pi)=0$, which shows that $\sfl(h,[0,\frac{3}{2}\pi])=\sgn\Gamma(L,\pi)=0$ by Proposition \ref{crossformcomsfl}. Note that by Theorem \ref{mainbifI} this is in accordance with our observation that there are no bifurcation radii.


\section*{Appendix}

\appendix

\section{Spectral flow and crossing forms}\label{sect-appsfl}
Let $W$ and $H$ be real Hilbert spaces with a dense injection $\iota:W\hookrightarrow H$. We denote by $\mathcal{L}(W,H)$ the Banach space of all bounded operators, and by $\mathcal{S}(W,H)$ the subset of all elements in $\mathcal{L}(W,H)$ which are selfadjoint when considered as operators on $H$ having dense domain $W$. We let $\mathcal{FS}(W,H)$ be the space of all selfadjoint Fredholm operators, and we recall that an operator in $\mathcal{S}(W,H)$ is Fredholm if and only if its kernel is of finite dimension and its image is closed. In what follows, we abbreviate $\mathcal{FS}(H):=\mathcal{FS}(H,H)$.\\
For a selfadjoint Fredholm operator $T_0\in\mathcal{FS}(W,H)$, there exists $\Lambda>0$ such that $\pm\Lambda$ do not belong to the spectrum 

\[\sigma(T_0)=\{\lambda\in\mathbb{R}:\,\lambda-T_0\,\text{not bijective}\}\] 
of $T_0$ and $\sigma(T_0)\cap[-\Lambda,\Lambda]$ consists only of isolated eigenvalues of finite multiplicity. We set for $-\Lambda\leq c<d\leq\Lambda$

\[E_{[c,d]}(T_0)=\bigoplus_{\lambda\in[c,d]}\ker(\lambda-T_0),\]
and we note that it is readily seen from the continuity of finite sets of eigenvalues (cf. \cite[\S I.II.4]{GohbergClasses}) that there exists a neighbourhood $N(T_0,\Lambda)\subset\mathcal{FS}(W,H)$ of $T_0$ such that $\pm\Lambda\notin\sigma(T)$ and $E_{[-\Lambda,\Lambda]}(T)$ has the same finite dimension for all $T\in N(T_0,\Lambda)$.\\
Let now $\mathcal{A}:[a,b]\rightarrow\mathcal{FS}(W,H)$ be a path of selfadjoint Fredholm operators having invertible endpoints. We choose a subdivision $a=t_0<t_1<\ldots<t_N=b$, operators $T_i\in\mathcal{FS}(W,H)$ and numbers $\Lambda_i>0$, $i=1,\ldots N$, such that the restriction of the path $\mathcal{A}$ to $[t_{i-1},t_i]$ runs entirely inside $N(T_i,\Lambda_i)$. The \textit{spectral flow} of $\mathcal{A}$ is defined by

\begin{align}\label{sfl}
\sfl(\mathcal{A},[a,b])=\sum^N_{i=1}{\dim E_{[0,\Lambda_i]}(\mathcal{A}_{t_i})-\dim E_{[0,\Lambda_i]}(\mathcal{A}_{t_{i-1}})}\in\mathbb{Z}.
\end{align}  
Note that, roughly speaking, $\sfl(\mathcal{A},[a,b])$ is the number of negative eigenvalues of $\mathcal{A}_a$ that
become positive as the parameter $t$ travels from $a$ to $b$ minus the number of positive eigenvalues of $\mathcal{A}_a$ that become negative, i.e., the net number of eigenvalues which cross zero.\\
Let us mention the following properties of the spectral flow, which we use throughout:

\begin{enumerate}
\item[i)] If $\mathcal{A}:[a,b]\rightarrow\mathcal{FS}(W,H)$ is a path and $\mathcal{A}_c$ invertible for some $c\in(a,b)$, then

	\begin{align*}
	\sfl(\mathcal{A},[a,b])=\sfl(\mathcal{A},[a,c])+\sfl(\mathcal{A},[c,b]).
	\end{align*}
\item[ii)] If $\mathcal{A}'$ is defined by $\mathcal{A}'_t=\mathcal{A}_{1-t}$ for some $\mathcal{A}:[a,b]\rightarrow\mathcal{FS}(W,H)$, then

\begin{align*}
\sfl(\mathcal{A}',[a,b])=-\sfl(\mathcal{A},[a,b]).
\end{align*}

\item[iii)] If $\mathcal{A}:[a,b]\rightarrow \mathcal{FS}(W,H)$ is such that $\mathcal{A}_t$ is invertible for all $t\in I$, then $\sfl(\mathcal{A},[a,b])=0$.

\item[iv)] Let $h:[0,1]\times [a,b]\rightarrow\mathcal{FS}(W,H)$ be a continuous map such that $h(s,a)$ and $h(s,b)$ are invertible for all $s\in[0,1]$. Then

\[\sfl(h(0,\cdot),[a,b])=\sfl(h(1,\cdot),[a,b]).\]

\item[v)] If 

\[\mu_{Morse}(\mathcal{A}_t):=\dim\left(\bigoplus_{\lambda<\infty}{\ker(\lambda-\mathcal{A}_t)}\right)<\infty,\quad t\in[a,b],\]
then 
\[\sfl(\mathcal{A},[a,b])=\mu_{Morse}(\mathcal{A}_a)-\mu_{Morse}(\mathcal{A}_b).\]
\end{enumerate}
The spectral flow of a continuously differentiable path $\mathcal{A}:[a,b]\rightarrow\mathcal{FS}(W,H)$ can be computed analytically. Let us denote by $\dot{\mathcal{A}}_{t_0}$ the derivative of $\mathcal{A}$ with respect to the parameter $t\in [a,b]$ at $t_0$. An instant $t_0\in(a,b)$ is called a \textit{crossing} if $\ker\mathcal{A}_{t_0}\neq 0$. The \textit{crossing form} at $t_0$ is the quadratic form defined by

\[\Gamma(\mathcal{A},t_0):\ker\mathcal{A}_{t_0}\rightarrow\mathbb{R},\,\,\Gamma(\mathcal{A},t_0)[u]=\langle\dot{\mathcal{A}}_{t_0}u,u\rangle_H,\]
and $t_0$ is called \textit{regular} if $\Gamma(\mathcal{A},t_0)$ is non-degenerate. The following two theorems can be found in \cite{crossings}, however, let us point out that in all their applications in the current paper, special cases that were proven before in \cite{RobbinSalamon} and \cite{FPR} are sufficient.

\begin{theorem}\label{thm:Sard}
There exists $\varepsilon>0$ such that 
\begin{itemize}
	\item[i)] $\mathcal{A}+\delta\, I_H$ is a path in $\mathcal{FS}(W,H)$ for all $|\delta|<\varepsilon$;
	\item[ii)] $\mathcal{A}+\delta\, I_H$ has only regular crossings for almost every $\delta\in(-\varepsilon,\varepsilon)$.
\end{itemize}
\end{theorem}

The next theorem shows that the spectral flow of $\mathcal{A}$ can be easily computed if all crossings are regular. 

\begin{theorem}\label{thm:crossing}
We assume that the path $\mathcal{A}$ has invertible endpoints. If $\mathcal{A}$ has only regular crossings, then they are finite in number and

\begin{align*}
\sfl(\mathcal{A},[a,b])=\sum_{t\in(a,b)}{\sgn\Gamma(\mathcal{A},t)},
\end{align*}
where $\sgn$ denotes the signature of a quadratic form.
\end{theorem}

Finally, let us recall from \cite{AtiyahSinger} the deep result that the space $\mathcal{FS}(H)$ of bounded selfadjoint Fredholm operators consists of three connected components

\[\mathcal{FS}(H)=\mathcal{FS}^+(H)\cup\mathcal{FS}^i(H)\cup\mathcal{FS}^-(H),\]
where

\[\mathcal{FS}^\pm(H):=\{T\in\mathcal{FS}(H):\,\mu_{Morse}(\pm T)<\infty\}\]
are contractible, and $\mathcal{FS}^i(H)$ is a classifying space for the functor $KO^{-7}$.
  

\section{The Maslov index in symplectic Hilbert spaces}\label{App:Maslov}

In this section we recall some facts about the {\em Fredholm Lagrangian Grassmannian\/} of a symplectic Hilbert space and the \textit{Maslov index}, where our basic reference is Furutani's work \cite{Furutani}.\\ 
Let $H$ be a real separable Hilbert space equipped with a symplectic form, that is, a skew-symmetric and non-degenerate bounded bilinear form $\omega$. Note that by definition, $\omega$ is non-degenerate if the canonical map $H\rightarrow H^\ast$, $u\mapsto \omega(\cdot,u)$ is bijective. For a subspace $\mu\subset H$, we use throughout the notation

\[\mu^\circ=\{u\in H:\,\omega(u,v)=0\,\,\text{for all}\,\, v\in \mu\}.\]

\begin{defi}\label{def:isoLagraangiansub}
A subspace $\mu$ of the symplectic Hilbert space $(H,\omega)$ is called {\em isotropic\/} if 
$\mu\subset\mu^\circ$, i.e., $\omega(u,v)=0$ for all $u,v \in \mu$.  If $\mu=\mu^\circ$, then $\mu$ is called {\em Lagrangian\/}.
\end{defi}

The set $\Lambda(H)$ of all {\em Lagrangian subspaces\/} of $H$ is a Banach manifold which is called the {\em Lagrangian Grassmannian\/} (cf. \cite[\S 1]{Nicolaescu}).\\
For $\dim H<\infty$ and a fixed Lagrangian subspace $\mu\in\Lambda(H)$, the Maslov index $\mu_{Mas}(\ell,\mu,[a,b])$ of a path $\ell:[a,b]\rightarrow\Lambda(H)$ was introduced in \cite{Arnold}, and heuristically, it counts non-transversal intersections of $\ell$ and $\mu$. Let us note for later reference the following well known example: Let $H=\mathbb{R}^{2k}$ and $\omega_0(x,y)=\langle\sigma u,v\rangle$, where $\sigma=\begin{pmatrix} 0&-I_k\\I_k&0\end{pmatrix}$ is the standard symplectic matrix and $\langle\cdot,\cdot\rangle$ denotes the standard scalar product in $\mathbb{R}^{2k}$. The product $H\times H$ is a symplectic space with respect to $\omega_2:=\omega_0\times(-\omega_0)$, and if $\ell:[a,b]\rightarrow\Sp(2k)$ is a path of symplectic matrices, then 

\[\gra\ell(t)=\{(u,\ell(t)u):u\in\mathbb{R}^{2k}\}\subset\mathbb{R}^{2k}\times\mathbb{R}^{2k},\quad t\in[a,b],\]
is a path of Lagrangian subspaces of $H\times H$. Let now $\mu,\mu'\in\Lambda(H)$ be two Lagrangian subspaces of $\mathbb{R}^{2k}$. Then $\mu\times\mu'\in\Lambda(H\times H)$ and $\mu_{Mas}(\gra\ell(\cdot),\mu\times\mu',[a,b])$ is defined. On the other hand, $\ell(t)\mu$, $t\in[a,b]$, is a path of Lagrangian subspaces in $H$ and so $\mu_{Mas}(\ell(\cdot)\mu,\mu',[a,b])$ exists as well. Clearly, $\gra\ell(\cdot)$ intersects $\mu\times\mu'$ non-transversally if and only if $\ell(\cdot)\mu$ intersects $\mu'$ non-transversally. Moreover, it turns out that also the corresponding Maslov indices coincide (cf. \cite[Thm. 3.2]{Robbin-SalamonMAS}):

\begin{align}\label{RobSalgra}
\mu_{Mas}(\gra\ell(\cdot),\mu\times\mu',[a,b])=\mu_{Mas}(\ell(\cdot)\mu,\mu',[a,b]).
\end{align}
One of the most important properties of the Maslov index $\mu_{Mas}(\ell,\mu,[a,b])$ is its invariance under homotopies having endpoints which are transversal to $\mu$. In contrast, it can be shown from Kuiper's theorem \cite{Kuiper} that $\Lambda(H)$ is a contractible space if $H$ is an infinite dimensional symplectic Hilbert space (cf. e.g. \cite[Prop. 1.1]{Nicolaescu}) and so no non-trivial homotopy invariant for paths in $\Lambda(H)$ can exist in this case. 

\begin{defi}\label{def:Fredholmpair}
Given two closed subspaces $\mu, \eta$ of $H$, the pair $(\mu, \eta)$ is called a {\em Fredholm pair\/} if 
\begin{align}\label{propertyFredholmpair}
\dim(\mu \cap \eta)<+\infty\quad\text{and}\quad \codim(\mu+\eta)<+\infty.
\end{align}
\end{defi}
Note that many authors require in the definition of a Fredholm pair also the sum $\mu+\eta\subset H$ to be closed, however, it is not hard to show that this property already follows from \eqref{propertyFredholmpair} (cf. \cite[\S IV.4.1]{Kato}).

\begin{defi}\label{def:FredLagGras}
The {\em Fredholm Lagrangian Grassmannian\/} with respect to the Lagrangian subspace $\mu \in \Lambda(H)$ is defined as 
\[
\FredL_\mu(H)=\{\eta \in \Lambda(H):\,(\mu, \eta)  \ \textrm{is a Fredholm pair\/}\},
\]
and the subset 
\[
\mathcal M_\mu(H)=\{\eta \in \FredL_\mu(H):\,\eta \cap \mu \neq \{0\}\},
\]
is called the {\em Maslov cycle\/} of $\mu$.
\end{defi}
Clearly, $\FredL_\mu(H)=\Lambda(H)$ if $\dim H<\infty$. Let $(H,\omega)$ be a symplectic Hilbert space and let $\mu\in\Lambda(H)$ be a fixed Lagrangian subspace. The construction of the Maslov index for paths $\ell:[a,b]\rightarrow\FredL_\mu(H)$ having endpoints outside of $\mathcal{M}_\mu(H)$ consists of two steps, which can be roughly described as follows (cf. \cite[\S 3.1]{Furutani}): First, a transformation of elements $\mu'\in\FredL_\mu(H)$ to unitary operators $\mathcal{U}(\mu')$ on a suitable Hilbert space is constructed, such that $\mu'\cap\mu\neq\{0\}$ if and only if the unitary operator $\mathcal{U}(\mu')$ corresponding to $\mu'$ has $-1$ in its spectrum. Second, one builds a spectral flow through $-1$ on the set of all paths in the image $\mathcal{U}(\FredL_\mu(H))$ whose endpoints do not have $-1$ in their spectra. The composition assigns an integer $\mu_{Mas}(\ell,\mu,[a,b])$ to every path $\ell:[a,b]\rightarrow\FredL_\mu(H)$ such that $\ell(a)\cap\mu=\ell(b)\cap\mu=\{0\}$, which has the following properties:

\begin{enumerate}
\item[i)] if $\ell(t) \notin \mathcal M_\mu(H)$ for each $t \in [a,b]$, then 

\[\mu_{Mas}(\ell,\mu, [a,b] )=0;\]

\item[ii)] if $\ell:[a,b] \to \FredL_\mu(H)$ is a continuous curve such that $\ell_c\notin\mathcal{M}_\mu(H)$ for some $c\in(a,b)$, then  
\[
\mu_{Mas}(\ell, \mu, [a,b])= \mu_{Mas}(\ell, \mu, [a,c]) +\mu_{Mas}(\ell,\mu, [c,b]);
\]
\item[iii)] if $H:[0,1]\times [a,b]\rightarrow\FredL_\mu(H)$ is continuous and $H(s,a),H(s,b)\notin\mathcal M_\mu(H)$ for all $s\in[0,1]$, then 

\[\mu_{Mas}(H(0,\cdot),\mu,[a,b])=\mu_{Mas}(H(1,\cdot),\mu,[a,b]).\]
\end{enumerate}

Finally, we recall from \cite[\S 3.4]{Furutani} the computation of the Maslov index by crossing forms. Let $\ell:[a,b] \to \FredL_{\mu}(H)$ be a $C^1$ path. We say that $t^\ast \in [a,b]$  is a {\em crossing instant \/} for the curve $\ell$, if $\ell(t^\ast) \in \mathcal M_\mu(H)$. If $\mu'$ is a Lagrangian subspace which is transversal to $\ell(t^\ast)$ at some crossing instant $t^\ast$, e.g. $\mu'=\ell(t^\ast)^\perp$, then there exists $\varepsilon>0$ such that $\ell(t)$ is transversal to $\mu'$ for each $|t-t^\ast|<\varepsilon$. Therefore, we can find a $C^1$-family of bounded operators $\phi_t: \gamma(t^\ast) \to \mu'$ such that
 
\[
\ell(t)=\gra(\phi_t),\quad t\in(t^\ast-\varepsilon,t^\ast+\varepsilon).
\]
The {\em crossing form\/} $\Gamma(\ell,\mu;t^\ast)$ at the instant $t=t^\ast$ is the quadratic form on $\gamma(t^\ast)\cap \mu$, defined by 
\[
\Gamma(\ell,\mu;t^\ast)[u]:= \dfrac{d}{dt}\Big\vert_{t=t^\ast} \omega (u, \phi_t(u)), \qquad u\in \gamma(t^\ast)\cap \mu.
\]
It can be shown that $\Gamma(\ell,\mu;t^\ast)$ does not depend on the choice of $\mu'$. A crossing $t^\ast \in (a,b)$ will be called {\em regular\/} if $\Gamma(\ell,\mu;t^\ast)$ is non-degenerate. It is easy to see that regular crossings are isolated and hence they are finite in number by the compactness of $[a,b]$. 

\begin{prop}\label{thm:localcontribution}
Let $\ell: [a, b] \to \FredL_\mu(H)$ be a $C^1$ path having endpoints outside $\mathcal{M}_\mu(H)$. If $\ell$ has only regular crossings, then

\[\mu_{Mas}(\ell,\mu, [a,b])= \sum_{t^\ast \in(a,b)} \sgn\Gamma(\ell,\mu;t^\ast),\]
where $\sgn$ denotes the signature. 
\end{prop}

We have not described the smooth structure on $\FredL_\mu(H)$ in this appendix for which we refer in particular to \cite[\S 2]{AbbMaj}. The following lemma is often useful for applying the previous proposition. Note that if $\ell:[a,b]\rightarrow\FredL_\mu(H)$ is a path, then there exists for every $\lambda\in[a,b]$ a unique orthogonal projection $P_\lambda\in\mathcal{L}(H)$ such that $\im P_\lambda=\ell(\lambda)$. 

\begin{lemma}\label{regcurve}
The path $\ell:[a,b]\rightarrow\FredL_\mu(H)$ is $C^l$, $l\in\{0,1,\ldots,\infty\}$, if and only if the associated path $P:[a,b]\rightarrow\mathcal{L}(H)$ of orthogonal projections is $C^l$.
\end{lemma}


\vspace{1cm}
Alessandro Portaluri\\
Department of Agriculture, Forest and Food Sciences\\
Universit\`a degli studi di Torino\\
Largo Paolo Braccini, 2\\
10095 Grugliasco (TO)\\
Italy\\
E-mail: alessandro.portaluri@unito.it

\vspace{1cm}
Nils Waterstraat\\
Institut für Mathematik\\
Humboldt-Universität zu Berlin\\
Unter den Linden 6\\
10099 Berlin\\
Germany\\
E-mail: waterstn@math.hu-berlin.de


\thebibliography{99}

\bibitem{AbbMaj} A. Abbondandolo, P. Majer, \textbf{Infinite dimensional Grassmannians}, J. Operator Theory \textbf{61}, 2009, 19--62

\bibitem{Ambrosetti} A. Ambrosetti, G. Prodi, \textbf{A Primer of Nonlinear Analysis}, Cambridge studies in advanced mathematics \textbf{34}, Cambridge University Press, 1993

\bibitem{Arnold} V.I. Arnol'd, \textbf{On a characteristic class entering into conditions of quantization}, Funkcional. Anal. i Prilo\u{z}en {\bf 1}, 1967, 1--14

\bibitem{AtiyahSinger} M.F. Atiyah, I.M. Singer, \textbf{Index Theory for Skew-Adjoint Fredholm Operators}, Inst. Hautes Etudes Sci. Publ. Math. \textbf{37}, 1969, 5--26



\bibitem{BoossFurutani} B. Booss-Bavnbek, K. Furutani, \textbf{The Maslov index: a functional analytical definition and the spectral flow formula}, Tokyo J. Math. \textbf{21},  1998, 1--34

\bibitem{BoossBuch} B. Booss-Bavnbek, K. Wojciechowski, \textbf{Elliptic boundary problems for Dirac operators}, Mathematics: Theory \& Applications, Birkh\"auser Boston, Inc., Boston, MA,  1993

\bibitem{Calderon} A.P. Calderon, \textbf{Uniqueness in the Cauchy problem for partial differential equations}, Amer. J. Math.  \textbf{80}, 1958, 16--36

\bibitem{AleDalbono} F. Dalbono, A. Portaluri, \textbf{Morse-Smale index theorems for elliptic boundary deformation problems}, J. Differential Equations  \textbf{253}, 2012, 463--480

\bibitem{DengJones} J. Deng, C.K.R.T. Jones, \textbf{Multi-dimensional Morse index theorems and a symplectic
view of elliptic boundary value problems}, Trans. Amer. Math. Soc. {\bf 363}, 2011, 1487--1508

\bibitem{DunfordSchwarz} N. Dunford, J.T. Schwartz, \textbf{Linear operators. Part II. Spectral theory. Selfadjoint operators in Hilbert space.}, Wiley Classics Library, A Wiley-Interscience Publication, John Wiley \& Sons, Inc., New York, 1988

\bibitem{Duistermaat} J.J. Duistermaat, \textbf{On the Morse Index in Variational Calculus}, Adv. Math. \textbf{21}, 1976, 173--195

\bibitem{FPR}  P.M. Fitzpatrick, J. Pejsachowicz, L. Recht, \textbf{Spectral Flow and Bifurcation of Critical Points of Strongly-Indefinite Functionals-Part I: General Theory}, J. Funct. Anal. \textbf{162}, 1999, 52--95

\bibitem{Frey} C. Frey, \textbf{On Non-local Boundary Value Problems for Elliptic Operators}, Dissertation, Universität zu Köln, 2005, http://kups.ub.uni-koeln.de/id/eprint/1512

\bibitem{Furutani} K. Furutani, \textbf{Fredholm Lagrangian Grassmannian and the Maslov index}, J. Geom. Phys. {\bf 51}, 2004, 269--331

\bibitem{GohbergClasses} I. Gohberg, S. Goldberg, M.A. Kaashoek, \textbf{Classes of linear operators. Vol. I}, Operator Theory: Advances and Applications \textbf{49}, Birkhäuser Verlag, Basel, 1990

\bibitem{Grubb} G. Grubb, \textbf{On coerciveness and semiboundedness of general boundary problems}, Israel J. Math. \textbf{10}, 1971, 32--95

\bibitem{Helfer} A.D. Helfer, \textbf{Conjugate Points on Spacelike Geodesics or Pseudo-Selfadjoint Morse-Sturm-Liouville Systems}, Pacific J. Math. \textbf{164}, 1994, 321--340

\bibitem{Hormander} L. Hörmander, \textbf{Linear partial differential operators}, Third revised printing, Die Grundlehren der mathematischen Wissenschaften, Band 116, Springer-Verlag New York Inc., New York, 1969

\bibitem{Kato} T. Kato, \textbf{Perturbation theory for linear operators}, Second edition, Grundlehren der Mathematischen Wissenschaften \textbf{132}, Springer-Verlag, Berlin-New York, 1976

\bibitem{Klingenberg} W. Klingenberg, \textbf{Riemannian Geometry}, de Gruyter, 1995

\bibitem{Kuiper} N.H. Kuiper, \textbf{The homotopy type of the unitary group of Hilbert space}, Topology {\bf 3}, 1965, 19--30


\bibitem{MilnorMorse} J.W. Milnor, \textbf{Morse Theory}, Princeton Univ. Press, 1969

\bibitem{MorseTrans} M. Morse, \textbf{The foundations of the calculus of variations in $m$-space (Part I)}, Trans. Amer. Math. Soc. \textbf{31}, 1929, 379--404

\bibitem{Morse} M. Morse, \textbf{The calculus of variations in the large}, Amer. Math. Soc. Colloq. Publ. \textbf{18}, 1934

\bibitem{MPP} M. Musso, J. Pejsachowicz, A. Portaluri, \textbf{A Morse Index Theorem for Perturbed Geodesics on Semi-Riemannian Manifolds}, Topol. Methods Nonlinear Anal. \textbf{25}, 2005, 69--99

\bibitem{Nicolaescu} L.I. Nicolaescu, \textbf{The Maslov index, the spectral flow, and decompositions of manifolds}, Duke Math. J. {\bf 80}, 1995, 485--533



\bibitem{JacoboIch} J. Pejsachowicz, N. Waterstraat, \textbf{Bifurcation of critical points for continuous families of $C^2$ functionals of Fredholm type}, J. Fixed Point Theory Appl. \textbf{13},  2013, 537--560, arXiv:1307.1043 [math.FA]

\bibitem{PalaisTerng} R.S. Palais, C-l Terng, \textbf{Critical Point Theory and Submanifold Geometry}, Springer-Verlag, 1988

\bibitem{AleIchDomain} A. Portaluri, N. Waterstraat, \textbf{On bifurcation for semilinear elliptic Dirichlet problems and the Morse-Smale index theorem}, J. Math. Anal. Appl. \textbf{408}, 2013, 572--575, arXiv:1301.1458 [math.AP]

\bibitem{AleIchBall} A. Portaluri, N. Waterstraat, \textbf{On bifurcation for semilinear elliptic Dirichlet problems on geodesic balls}, J. Math. Anal. Appl. \textbf{415}, 2014, 240--246, arXiv:1305.3078 [math.AP]

\bibitem{PiccioneMITinSRG} P. Piccione, D.V. Tausk, \textbf{The Morse Index Theorem in Semi-Riemannian Geometry}, Topology \textbf{41}, 2002, 1123-1159, arXiv:math/0011090

\bibitem{Robbin-SalamonMAS} J. Robbin, D. Salamon, \textbf{The Maslov index for paths}, Topology \textbf{32}, 1993, 827--844

\bibitem{RobbinSalamon} J. Robbin, D. Salamon, \textbf{The spectral flow and the {M}aslov index}, Bull. London Math. Soc. {\bf 27}, 1995, 1--33



\bibitem{Smale} S. Smale, \textbf{On the {M}orse index theorem}, J. Math. Mech. \textbf{14}, 1965, 1049--1055

\bibitem{SmaleCorr} S. Smale, \textbf{Corrigendum: ``{O}n the {M}orse index theorem''}, J. Math. Mech. \textbf{16}, 1967, 1069--1070

\bibitem{Swansona} R.C. Swanson, \textbf{Fredholm intersection theory and elliptic boundary deformation problems I}, J. Differential Equations \textbf{28}, 1978, 189--201

\bibitem{Swansonb} R.C. Swanson, \textbf{Fredholm intersection theory and elliptic boundary deformation problems II}, J. Differential Equations \textbf{28}, 1978, 202--219

\bibitem{Uhlenbeck} K. Uhlenbeck, \textbf{The Morse index theorem in Hilbert space}, J. Differential Geometry \textbf{8}, 1973, 555--564

\bibitem{Wa} N. Waterstraat, \textbf{A K-theoretic proof of the Morse index theorem in semi-Riemannian Geometry}, Proc. Amer. Math. Soc. \textbf{140}, 2012, 337--349

\bibitem{WaterstraatSurvey} N. Waterstraat, \textbf{On bifurcation for semilinear elliptic Dirichlet problems on shrinking domains}, Elliptic and Parabolic Equations (J. Escher, J. Seiler, C. Walker (Eds.)), Springer Proc. Math. Stat. \textbf{119}, 2015, 279--298, arXiv:1403.4151 [math.AP]

\bibitem{crossings} N. Waterstraat, \textbf{Spectral flow, crossing forms and homoclinics of Hamiltonian systems}, submitted, 35 pp., arXiv:1406.3760 [math.DS]

\bibitem{Weidmann} J. Weidmann, \textbf{Linear Operators in Hilbert Spaces}, Graduate Texts in Mathematics \textbf{68}, Springer-Verlag, 1980

\bibitem{Zeidler} E. Zeidler, \textbf{Nonlinear functional analysis and its applications. II/A. Linear monotone operators}, Springer-Verlag, 1990
\end{document}